\documentclass[11pt]{article}
\usepackage{amsthm,amsmath,amssymb,amsfonts}
\usepackage{color}
\usepackage{graphicx}
\usepackage{euscript,mathrsfs}

\newcommand{\dx}{{\, \rm d}x}

\newcommand{\Div}{{\rm div}\,}
\newtheorem{thm}{Theorem}

\newtheorem{prop}[thm]{Proposition}
\newtheorem{df}{Definition}

\newcommand{\Ov}[1]{\overline{#1}}

\newcommand{\pder}[2]{\frac{\partial #1}{\partial #2}}

\newcommand{\vr}{\varrho}

\newcommand{\vrd}{\vr_\delta}

\newcommand{\vud}{\vu_\delta}

\newcommand{\vu}{\vc{u}}

\newcommand{\vc}[1]{{\bf #1}}
\newcommand{\vcg}[1]{{\pmb #1}}
\newcommand{\F}[1]{$\mathbb{#1}$}
\newcommand{\Grad}{\nabla}

\newcommand{\tn}[1]{\mbox {\F #1}}
\newcommand{\N}{\tn{N}}
\newcommand{\dt}{\, {\rm d} t }

\newcommand{\intO}[1]{\int_{\Omega} #1\dx}

\newcommand{\ep}{\varepsilon}

\begin{document}

\title{Weak solutions for some compressible multicomponent fluid models}
\author{Anton\'\i n  Novotn\'y$^1$ and  Milan Pokorn\' y$^2$}
\maketitle

\bigskip

\centerline{$^1$ Institut de Math\'ematiques de Toulon, EA 2134}
\centerline{BP20132, 83957 La Garde, France}
\centerline{e-mail: {\tt novotny@univ-tln.fr}}

\centerline{$^{2}$ Charles University, Faculty of Mathematics and Physics}
\centerline{Mathematical Inst. of Charles University, Sokolovsk\' a 83, 186 75 Prague 8, Czech Republic}
\centerline{e-mail: {\tt pokorny@karlin.mff.cuni.cz}}
\vskip0.25cm
\begin{abstract}
The principle purpose of this work is to investigate a "viscous" version of a "simple"
but still realistic bi-fluid model described in \cite[Bresch, Desjardin, Ghidaglia, Grenier, Hillairet]{BreschMF} whose "non-viscous" 
version is derived from physical considerations in \cite[Ishii, Hibiki]{ISHI} as a particular sample of a multifluid model with 
algebraic closure. The goal is to show existence of weak solutions for large initial data on an arbitrarily large time interval.
We achieve this goal by transforming the model to an academic system which resembles to the compressible Navier--Stokes equations,
with however two continuity equations and a momentum equation endowed with pressure of complicated structure dependent on two variable densities.
The new "academic system" is then solved by an adaptation of the Lions--Feireisl approach for solving compressible Navier--Stokes equation, 
completed with several observations related to the DiPerna--Lions transport theory inspired by \cite[Maltese, Mich\'alek, Mucha, Novotn\'y, Pokorn\'y, Zatorska]{3MNPZ}
and \cite[Vasseur, Wen, Yu]{VWY}. We also explain how these techniques can be generalized to a model of mixtures with more then two species.

This is the first result on the existence of weak solutions for any realistic multifluid system.

\end{abstract}

\noindent{\bf MSC Classification:} 76N10, 35Q30

\smallskip

\noindent{\bf Keywords:} bi-fluid system; multifluid system; compressible Navier--Stokes equations; transport equation; continuity equation; large data weak solution.

\section{Introduction} \label{se1}

The rigorous mathematical results on existence of weak solutions in large for realistic multi-fluid models in more than one space dimension are very rare (if not non existing) in the mathematical literature.
One of the most simple bi-fluid models, still realistic in some physical situations, is a model of two compressible fluids with common  velocity and dissipation  described in Bresch et al. \cite{BreschMF}. Its  "non-viscous" counterpart can be formally obtained from more complex two velocity models by a process of interface  averaging and convenient algebraic closure. We refer the reader to the seminal works of Ishii and Hibiki \cite{ISHI} and of D. Drew and S.L. Passman \cite{DRPAS} for more details and exposition on the different models and the averaging process in multi-fluid modeling, and to works of Evje \cite{EV1}, \cite{EV2} for some
existence results for one-dimensional multifluid models.

The compressible bi-fluid model provided in Bresch et al. \cite[Section 2.2.3]{BreschMF} reads:
\begin{equation}\label{eq1.1bi}
\begin{aligned}
\partial_t (\mathfrak{a}\vr_+)  + \Div(\mathfrak{a}\vr_+\vu) &= 0, \\
\partial_t ((1-\mathfrak{a})\vr_-)  + \Div((1-\mathfrak{a})\vr_-)\vu) &= 0,  \\
\partial_t \big((\mathfrak{a}\vr_++(1-\mathfrak{a})\vr_-))\vu\big) &+\\ \Div\Big(\big(\mathfrak{a}\vr_++(1-\mathfrak{a})\vr_-\big) \vu\otimes \vu\Big) + \nabla \mathfrak{P}_+(\vr_+) &= \mu \Delta \vu + (\mu+\lambda)\Grad \Div \vu,\\
\mathfrak{P}_{+}(\vr_+) &= \mathfrak{P}_{-}(\vr_-),
\end{aligned}
\end{equation}
where $\mathfrak{P}_{\pm}$  are given functions characterizing the species in the mixture while $0\le \mathfrak{a}\le 1$, $\vr_+\ge 0$,
$\vr_-\ge 0$ and $\vc u$ are unknown functions. They have the following meaning:
$\mathfrak{a}$, $\mathfrak{a}\vr_+$, $(1-\mathfrak{a})\vr_+$ denote the rate of amount of the first species, density of the first and the second species in the mixture, respectively, while $\vu$ is the velocity of the mixture. The constants $\mu$ and $\lambda$ are the shear and bulk (average) viscosities of the mixture,   We assume $\mu>0$ and $2\mu + 3\lambda \geq 0$ as is standard and physically reasonable in such situation.

The above equations are set in the time-space domain $Q_T=I\times \Omega$, where $\Omega$ is a sufficiently smooth bounded domain of $R^3$ and $I=(0,T)$ is the time interval with $T>0$ arbitrary large. 

The system is endowed  with boundary condition
\begin{equation} \label{eq1.2bi}
\vu = \vc{0}
\end{equation}
on $(0,T)\times \partial \Omega$, and the initial conditions in $\Omega$
\begin{equation} \label{eq1.3bi}
\begin{aligned}
\mathfrak{a}\vr_+(0,x) &= \mathfrak{a}_0\vr_{+,0}(x):=\vr_0(x), \\
(1-\mathfrak{a})\vr_-(0,x) &= (1-\mathfrak{a}_0) \vr_{-,0}(x):=Z_0(x), \\
(\mathfrak{a}\vr_++(1-\mathfrak{a})\vr_-)\vu(0,x)&= 
(\mathfrak{a}_0\vr_{0,+}+(1-\mathfrak{a}_0)\vr_{-,0})\vu_0(x):=\vc M_0.
\end{aligned}
\end{equation}

In this paper, we will refer to this system as to the {\it real bi-fluid system}.

One of the goals of this work is to prove existence of weak solutions for this model
under very mild (and quite realistic) assumptions on constitutive functions
$\mathfrak{P}_{\pm}$. The proof will be based on the reformulation of the original problem via a change of variables to a more convenient problem of PDEs. The new 
reformulated boundary value problem reads: 

\begin{equation}\label{eq1.1}
\begin{aligned}
\partial_t \vr  + \Div(\vr \vu) &= 0, \\
\partial_t Z + \Div (Z \vu) & = 0,  \\
\partial_t \big((\vr+Z)\vu\big) + \Div\big((\vr+Z) \vu\otimes \vu) + \nabla P(\vr,Z) &= \mu \Delta \vu + (\mu+\lambda)\Grad \Div \vu.
\end{aligned}
\end{equation}
This system reminds the compressible Navier--Stokes equations; however, the pressure $P=P(\vr,Z)$ depends on two densities (possibly in a complicated way). Both partial densities $\vr$ and $Z$ satisfy the continuity equation (without any source or diffusive terms) and the total density (the sum of them) appears in the inertial terms in the momentum equation. 

As for system (\ref{eq1.1bi}) we complete the new system by the boundary condition
\begin{equation} \label{eq1.2}
\vu = \vc{0}
\end{equation}
on $(0,T)\times \partial \Omega$ and the initial conditions in $\Omega$
\begin{equation}\label{eq1.3}
\vr(0,x) = \vr_0(x), \;
Z(0,x) = Z_0(x), \;
(\vr + Z)\vu(0,x) = \vc{m}_0(x).
\end{equation}

This system is of independent interest. In this paper, we will call it a {\it academic
multifluid system}. 

In fact, it seems to reflect the essence of mathematical difficulties allowing to mimic --- especially in the case of complex dependence of pressure on densities $\vr$ and $Z$ --- the fundamental properties of not only  the bi-fluid system (\ref{eq1.1bi}) but also of other multi-fluid models, as e.g. one velocity Baer--Nunziato model without relaxation.
System (\ref{eq1.1}--\ref{eq1.3}) has been recently studied for simple "toy" pressure functions  by A. Vasseur at al. \cite{VWY} --- with $P(\vr,Z) = \vr^\gamma + Z^\beta$ for some $\gamma >\frac 95$ and $\beta \geq 1$ and in \cite{3MNPZ} for 
$P(Z)=Z^\gamma$, $\gamma>3/2$ (where, however, $\vr+Z$ is replaced by $\vr$ in the inertial terms). Although the primal purpose of \cite{3MNPZ} is to investigate the compressible Navier--Stokes equations with entropy transport rather
than multifluid flows, both approaches present numerous similar features. One of essential  ingredients of proofs in both papers --- besides the application of the Lions--Feiresl compactness approach --- is a particular interrelation of quantities obeying renormalized continuity and transport equations. These properties will be essential as well in this paper.
 
Our approach allows  to treat more than two species. However, the two-phase model itself under the general conditions on the pressure is already quite complicated, and assuming directly more complex situation would make the paper significantly less understandable.  We therefore prefer to keep at the beginning the more simple situation of solely two fluids, and at the end, in the last section, we just briefly formulate the more general problem for many species and explain the differences in the proof; however, we treat only the academic multifluid problem, as the real multifluid problem becomes quite technical and its treatment 
would extend the length of the paper considerably.

The weak solutions to both systems are defined in the next Section (see Definitions
\ref{d1} and \ref{d1bi}), and the exact existence statements are announced in Theorem \ref{t1} (dealing with the academic system (\ref{eq1.1}--\ref{eq1.3})) and in Theorem \ref{t1bi}
(dealing with the real bi-fluid system (\ref{eq1.1bi}--\ref{eq1.3bi})). 

Theorem \ref{t1} generalizes
the result of Vasseur et al. \cite{VWY} in several directions: 1) We allow complicated (even non-monotone) pressure functions (cf. assumption (H3--H5) in the next section);
this permits us to treat problems beyond purely academic examples. 2) In contrast
with \cite{VWY} we can consider also the borderline case, which is in this situation so far $\gamma=9/5$. We allow exponents $\beta>0$ instead of $\beta\ge 1$.
3) We may consider more then two species, cf. Theorem \ref{t2}.  
 
The best expected result within the borderlines of the existing mono-fluid theory would be $\gamma>3/2$. This remains still  an open problem.

Theorem \ref{t1bi}
is to the best of our knowledge  the first result  for a fully non-linear bi-fluid model really used in compressible multi-fluid modeling. It applies to strictly increasing pressure laws $\mathfrak{P}_\pm$ with growth at infinity corresponding to $\gamma^+\ge 9/5$,  $\gamma^->0$ and a certain (mild) restriction  on the growth
of $\mathfrak{P}_\pm$ near $0$ imposed by conditions (\ref{Hbi3+}--\ref{Hbi4}) and on $\gamma^{-}$ large with respect to $\gamma^+$ due to (\ref{gamma}), cf. Remark \ref{se2}.3. In particular, the assumptions are satisfied for isentropic pressure laws
$\mathfrak{P}_{\pm}(\vr_\pm) = \vr_\pm^{\gamma_\pm}$, $\gamma^+\ge 9/5$,   $\gamma^- > \frac{\gamma^+}{\sqrt{\gamma^+ +1}}$ (where the latter constraint is 
due to condition (\ref{Hbi3+}) imposing limitation on the power law near $0$) and $\gamma^{-} + \frac{\gamma^{-}}{\gamma^{+}}-\frac{\gamma^{+}}{\gamma^{-}} <\gamma^{+} + \min\{\frac 23\gamma^{+} -1, \frac{\gamma^{+}}{2}\}$ if $\gamma^->\gamma^+$. Even weaker restrictions can be obtained if $\underline a >0$, see (\ref{eq2.1}) below.  
The best expected limitation for the growth at infinity corresponding to the mono-fluid compressible flows, namely $\gamma^+>3/2$, has not been achieved yet. It is again an excellent open problem.

It is to be noticed that the linearized semi-stationary version of this system (composed of two continuity equations of the momentum equation with neglected material derivative) has been investigated very recently in Bresch et al. \cite{BRZA} with particular pressure laws $\mathfrak{P}_{\pm}(\vr_\pm)= \vr_\pm^{\gamma_\pm}$, $\gamma_\pm>1$ on a 3-D periodic cell. This paper has, besides \cite{VWY} and \cite{3MNPZ}, largely inspired the present paper. Note, however, that the authors did not use the Feireisl--Lions approach to prove the compactness of the density, but they applied the idea due to D. Bresch and P.E. Jabin, see \cite{BJ}.

In what follows, the scalar-valued functions will be printed with the usual font, the vector-valued functions will be printed in bold, and the tensor-valued functions with a special font, i.e. $\vr$ stands for the density, $\vu$ for the velocity field and $\tn{S}$ for the stress tensor.  
We use standard notation for the Lebesgue and Sobolev spaces equipped by the standard norms $\|\cdot\|_{L^p(\Omega)}$ and $\|\cdot\|_{W^{k,p}(\Omega)}$, respectively. We will sometimes distinguish the scalar-, the vector- and the tensor-valued functions in the notation, i.e. we use $L^p(\Omega)$ for scalar quantities, $L^p(\Omega;R^3)$ for vectors and $L^p(\Omega;R^{3\times 3})$ for tensors. The indication of the $R$ or tensor character of the fields (here $;R^3$ or $;R^{3\times3}$) may be omitted, when there is no lack of confusion. The Bochner spaces of integrable functions on $I$ with values in a Banach space $X$ will be denoted $L^p(I;X)$; 
likewise the spaces of continuous functions on $\overline I$ with values in $X$ will be denoted $C(\overline I;X)$. The norms in the Bochner spaces will be denoted $\|\cdot\|_{L^p(I;X)}$ and $\|\cdot\|_{C(\overline I;X)}$, respectively. In most cases, the Banach space $X$ will be either the Lebesgue or the Sobolev space. 
Finally, we use vector spaces $C_{\rm weak}(\overline I; X)$ of continuous  functions in $\overline I$  with respect to weak topology of $X$ (meaning
that $f\in C_{\rm weak}(\overline I; X)$ iff $t\mapsto {\cal F}(f(t))$ belongs for any ${\cal F}\in X^*$ to $C( \overline I)$).

 The generic constants will be denoted by $c$, $\underline c$, $\overline c$, $C$, 
$\underline C$, $\overline C$ and their value may change even in the same formula or in the same line.

\section{Assumptions and main results}\label{se2}

In this section, we shall list and motivate our assumptions, and state
the main results.

As usual for this type of equations, we shall assume throughout the paper that
\begin{equation}\label{Om}
\mbox{$\Omega$ is a bounded domain of class $C^{2,\nu}$, $\nu\in (0,1)$.}
\end{equation}

We will, exactly as in \cite{VWY} and \cite{3MNPZ},   use the properties of the regularized continuity equation which allow us to obtain a certain type of a minimum principle. We therefore assume: \\ \\
{\bf Hypothesis (H1).}
\begin{equation} \label{eq2.1}
(\vr_0,Z_0)\in {\cal O}_{\underline a}:=
\{(\vr,Z)\in R^2\,|\,\vr\in [0,\infty),\underline a\vr 
\le Z\le \overline a \vr\},\;\mbox{$0\le\underline a 
<\overline a<\infty$}.
\end{equation}
In what follows we will always use the following convention for the calculus of fractions $\frac Z\vr$ provided $(\vr,Z)\in {\cal O}_{\underline a}$,
namely
\begin{equation}\label{conv}
s=\frac Z\vr:=\left\{
\begin{array}{c} \frac Z\vr \;\mbox{if $\vr>0$},\\
\;\mbox{if $\vr=0$}.
\end{array}
\right.
\end{equation}


We impose natural conditions on the integrability of the initial data: \\ \\
{\bf Hypothesis (H2).}
\begin{equation} \label{eq2.6}
\vr_0 \in L^\gamma(\Omega), \; Z_0 \in L^\beta(\Omega) \; \text{ if } \beta > \gamma,
\end{equation}
$$
\vc{m}_0 \in L^1(\Omega;R^3),\;
(\vr_0+Z_0)|\vu_0|^2\in L^1(\Omega).
$$

{ Next three sets of hypotheses deal with the constitutive law of pressure.
They are intrinsically related 
to the Helmholtz free energy function $H_P$, which appears naturally in the (formal) energy identity of the system (\ref{eq1.1}--\ref{eq1.3}). 

}

{The Helmholtz free energy function 
$H_P(\vr,Z)$ corresponding to $P$ is a  solution of the partial differential equation of the first order in $(0,\infty)^2$
\begin{equation}\label{HPODE}
P(\vr,Z) = \vr \pder{H_P(\vr,Z)}{\vr} + Z\pder{H_P(\vr,Z)}{Z}-H_P(\vr,Z).
\end{equation}
It is not uniquely determined. However, we can find one of its admissible explicit solutions by using the method of characteristics, namely
\begin{equation}\label{HP}
H=H_P(\vr,Z)=\vr\int_1^\vr \frac {P(s,s\frac Z\vr)}{s^2}\,{\rm d}s
{\;\mbox{if $\vr>0$},\; H_P(0,0)=0.}
\end{equation}

}
\noindent
{\bf Hypothesis (H3).} \\ \\
We suppose that pressure $P$: $[0,\infty)^2\to [0,\infty)$, P(0,0)=0, $P\in C^1((0,\infty)^2)$ is such that
\begin{equation}\label{it1-}
\forall Z\ge 0,\;\mbox{function $\vr\mapsto P(\vr,Z)$ is continuous in $[0,\infty)$},
\end{equation}
$$
\forall \vr\ge 0,\;\mbox{function $Z\mapsto P(\vr,Z)$ is continuous in $[0,\infty)$}
$$ 
and
there is a positive constant $C$ such that for all $(\vr,Z)\in {\cal O}_{\underline a}$
 \begin{equation} \label{eq2.2}
 \underline C(\vr^\gamma + Z^\beta -1)\le P(\vr,Z) \leq \overline C(\vr^\gamma + Z^\beta +1), \\
\end{equation}
with $\gamma \geq \frac{9}{5}$, $\beta > 0$ and $\gamma_{BOG}= \min\{\frac 23 \gamma-1,\frac{\gamma}{2}\}$, the improvement of the integrability due to the Bogovskii operator estimates, cf. Section \ref{se7}.

We moreover assume 
\begin{equation}\label{eq2.5-}
|\partial_ZP(\vr,Z)|\le C(\vr^{-\underline\Gamma}+\vr^{\overline\Gamma-1})\;\mbox{in ${\cal O}_{\underline a}\cap (0,\infty)^2$}
\end{equation}
with some $0\le\underline\Gamma<1$, and with some $0< \overline\Gamma < \gamma + \gamma_{BOG}$ if $\underline a=0$, 
$0<\overline\Gamma<{\rm max}\{\gamma+\gamma_{BOG}, \beta+\beta_{BOG}\} $ if $\underline a>0$. 
\\ \\
{\bf Hypothesis (H4).}\\ \\
Next we assume
%
\begin{equation} \label{eq2.4}
P(\vr,\vr s)=
{\cal P}(\vr,s) - {\cal R} (\vr,s),
\end{equation}
where $[0,\infty)\ni\vr\mapsto {\cal P}(\vr,s)$ is non decreasing  for any $s\in [\underline a,\overline a]$, and $\vr\mapsto {\cal R}(\vr,s)$ is for any $s \in [\underline a,\overline a]$ a non-negative $C^2$-function  in $[0,\infty)$ uniformly bounded with respect to $s \in [\underline a, \overline a]$ with compact support uniform with respect to $s \in [\underline a, \overline a]$ . Here, $\underline a, \overline a$ are the constants from relation (\ref{eq2.1}). Moreover, if $\gamma=\frac 95$, we suppose that
\begin{equation}\label{?!}
{\cal P}(\vr,s)=f(s)\vr^\gamma +\pi(\vr,s),
\end{equation}
where $[0,\infty)\ni\vr\mapsto \pi(\vr,s)$ is non decreasing  for any $s\in [\underline a,\overline a]$ and $f\in L^\infty(\underline a,\overline a)$, ${\rm ess\, inf}_{s\in(\underline a,\overline a)} f(s)\ge \underline f>0$. {Finally, we shall assume
\begin{equation}\label{?!+}
\forall\vr\in (0,1),\; \sup_{s\in [\underline a,\overline a]}P(\vr,\vr s)\le c\vr^\alpha\;\mbox{with some $c>0$ and $\alpha>0$}.
\end{equation}
}

{The last hypothesis is technical and few restrictive.}\\ \\
\vbox{\noindent 
{\bf Hypothesis (H5).}\\ \\
Function $\vr\mapsto P(\vr,Z)$ is for all $Z>0$ locally Lipschitz on $(0,\infty)$  and function $Z\mapsto \partial_Z P(\vr,Z)$ is for all $\vr>0$ locally Lipschitz on $(0,\infty)$ with Lipschitz constant
\begin{equation}\label{eq2.3a-}
\widetilde L_P(\vr,Z)\le C(\underline r)(1+\vr^A)\;\mbox{ for all $\underline r>0$, $(\vr, Z)\in {\cal O}_{\underline a} \cap (\underline r,\infty)^2$} 
\end{equation}
with some non negative number $A$. Number $C(\underline r)$ may diverge to $+\infty$ as $\underline r\to 0^+$.}
\\ \\
\noindent
{\bf Remark \ref{se2}.1}
{\it
\begin{description}
\item{\it 1.} As formulated more precisely in Section \ref{se3}, assumption (\ref{eq2.1}) allows to show that $(\vr(t,x), Z(t,x))$ in ${\cal O}_{\underline a}$ for $(t,x) \in (0,\infty)\times \Omega$ via a certain type of minimum principle.

\item{\it 2.} Assumption (\ref{eq2.2}) in Hypothesis (H3) is used to prove estimates
of the densities (and also of the pressure). We explain its use in Sections \ref{se5}
and \ref{se7}. To this end it is important to notice that it implies the following bounds for 
the corresponding Helmholtz function $H_P$ defined in (\ref{HP}),
\begin{equation} \label{eq2.3}
\underline C (\vr^\gamma+Z^\beta -1)\le H_P(\vr,Z) \leq \overline C(\vr^\gamma + Z^\beta +1)
\end{equation}
 for all $(\vr,Z)\in {\cal O}_{\underline a}$.
 It is this  estimate  which primarily implies the bounds for densities via the energy inequality (regardless bounds (\ref{eq2.2})). Consequently, the upper bound (\ref{eq2.2})
guarantees the integrability of pressure, and the lower bound (\ref{eq2.2})
is employed in conjunction with the Bogovskii operator to improve the estimates for densities, cf. Sections \ref{se3}, \ref{se6}, \ref{se7}. 
\item{\it 3.} Assumption (\ref{eq2.5-}) implies that function $s\mapsto P(\vr,\vr s)$ defined in Hypothesis (H4) is Lipschitz on $[\underline a,\overline a]$ with Lipschitz constant 
\begin{equation} \label{eq2.5}
L_{P}(\vr) \leq C\big(\vr^{1-\underline{\Gamma}} 1_{\{0\leq \vr \leq 1\}} + \vr^{\overline{\Gamma}} 1_{\{\vr \geq 1\}}\big)
\end{equation}
with $\underline{\Gamma} {\in [0,1)}$ and $\overline{\Gamma} < \gamma + \gamma_{BOG}$ or $\beta +\beta_{BOG}$ (if $\underline{a}>0$ and $\beta >\gamma$).
This fact will be used in each step of the convergence proof in conjunction
with the fact that the quantity $\vr s^2$, where $s=Z/\vr$, verifies continuity
equation (which implies a certain form of compactness for $s$ itself, cf. Proposition
 \ref{L2}). Hypothesis (H1) in conjunction
with the maximum principle applied to the approximations of both continuity
equations will ensure  $\underline a\leq s\leq \overline a$ (see Section \ref{se5} for more details), which motivates the requirement on the range of $s$ in (\ref{eq2.5}) and range of $Z$ in estimates (\ref{eq2.2}), (\ref{eq2.3}), (\ref{eq2.3a}), (\ref{eq2.5-}), (\ref{eq2.3a-}). {Assumption (\ref{?!+}) on behaviour of
function $\vr\mapsto P(\vr,\vr s)$ near zero is not very much restrictive. It ensures the continuity of Helmholtz function $H_P$ at point $(0,0)$. It is easy to verify
that $H_P\in C(\overline{{\cal O}_{\underline a}})$ (and $H_P(0,0)=0$).}
\item{\it 4.} Decomposition (\ref{eq2.4}) is crucial for the proof of strong convergence of density via the Lions' method. It is inspired by Feireisl \cite{Fe2002}, where a similar decomposition has been applied in order to prove existence of a weak solutions to the compressible Navier--Stokes equations with possibly non-monotone pressure (for values of the density from a bounded interval).
\item{\it 5.} Assumption (\ref{eq2.3a-}) is technical and little restrictive. It is
related to the way how the equations are approximated and will be used to derive
estimates at the first level of approximations, cf. Section \ref{se5}. To this end,
it is important to notice that it implies estimate
\begin{equation} \label{eq2.3a}
\Big|\frac{\partial ^2 H}{\partial \vr^2}(\vr,Z)\Big| +  \Big|\frac{\partial ^2 H}{\partial \vr \partial Z}(\vr,Z)\Big| + \Big|\frac{\partial ^2 H}{\partial Z^2}(\vr,Z)\Big| \leq C(\underline r)(1+\vr^A)
\end{equation} 
for all $\underline r>0$, $(\vr, Z)\in {\cal O}_{\underline a} \cap (\underline r,\infty)^2$
\end{description}
as one can compute directly from the explicit formula (\ref{HP}).
}
\\ \\
\noindent
{\bf Remark \ref{se2}.2}
{\it
Before presenting the weak formulation of our problem and formulating our main existence results, let us show several examples of the pressure functions which fulfill  Hypotheses (H3--H5).
\begin{description}
\item{\it 1.}
We may take
\begin{equation} \label{e1}
P(\vr,Z) = \vr^\gamma + Z^\beta + \sum_{i=1}^M F_i(\vr,Z),
\end{equation}
where $F_i (\vr,Z) = C_i \vr^{r_i} Z^{s_i}$, $0\leq r_i<\gamma$, $0 \leq s_i <\beta$, $r_i +s_i <\max \{\gamma, \beta\}$. It is an easy matter to check that all Hypotheses (H3--H5) are fulfilled.
\item {\it 2.}
Another possibility, not covered in \cite{VWY}, is 
\begin{equation} \label{e2}
P(\vr,Z) = (\vr+Z)^\gamma + \sum_{i=1}^M F_i(\vr,Z),
\end{equation}
where $F_i$ are as above (for $\beta =\gamma$). Again, it is easy to verify all Hypotheses (H3--H5).
\item{\it 3} A nontrivial example is pressure defined in (\ref{biP}) corresponding to the alternative formulation of system (\ref{eq1.1bi}) in the form (\ref{eq1.1}).
It is proved in Section \ref{sdod1} that it satisfies Hypotheses ({H3}--{H5})
under assumptions (\ref{Hbi2}--\ref{Hbi4}).
\end{description}
}

We now explain the notion of the weak solution to  problem (\ref{eq1.1}--\ref{eq1.3}).

\begin{df}\label{d1}
The triple $(\vr,Z,\vu)$ is a bounded energy weak solution to problem
(\ref{eq1.1}--\ref{eq1.3}), if $\vr, Z\geq 0$ a.e. in $I\times \Omega$, $\vr \in L^\infty(I;L^\gamma(\Omega))$, $Z\in L^\infty(I;L^{\gamma}(\Omega))$, $\vu \in L^2(I;W^{1,2}_0(\Omega;R^3))$, $(\vr+ Z)|\vu|^2 \in L^\infty(I;L^1(\Omega))$, $P(\vr,Z) \in L^1(I\times \Omega)$, and
\begin{equation} \label{eq2.7}
\begin{aligned}
\int_0^T \int_\Omega \big(\vr \partial_t \psi + \vr \vu \cdot \Grad \psi\big) \dx \dt  + \int_\Omega \vr_0 \psi(0,\cdot) \dx &=0 \\
\int_0^T \int_\Omega \big(Z \partial_t \psi + Z \vu \cdot \Grad \psi\big) \dx \dt + \int_\Omega Z_0 \psi(0,\cdot) \dx &=0
\end{aligned}
\end{equation}
for any $\psi \in C^1_c([0,T) \times \Ov{\Omega})$, 
\begin{equation} \label{eq2.8}
\begin{aligned}
&\int_0^T \int_\Omega \big((\vr+Z)\vu \cdot \partial_t \vcg{\varphi} + (\vr+Z) (\vu\otimes \vu): \Grad \vcg{\varphi} + P(\vr,Z) \Div \vcg{\varphi}\big) \dx \dt \\
= &\int_0^T\int_\Omega (\mu \Grad \vu :\Grad \vcg{\varphi} + (\mu+\lambda) \Div \vu \, \Div \vcg{\varphi} \big) \dx \dt - \int_\Omega \vc{m}_0 \cdot \vcg{\varphi}(0,\cdot) \dx
\end{aligned} 
\end{equation}
for any $\vcg{\varphi} \in C^1_c([0,T) \times \Omega;R^3)$, and the energy inequality holds
\begin{equation} \label{eq2.9}
\begin{aligned} 
&\int_\Omega \Big(\frac 12 (\vr+Z)|\vu|^2 + H_P(\vr,Z)\Big)(\tau,\cdot) \dx  \\
+& \int_0^\tau \int_\Omega \big(\mu|\Grad \vu|^2 + (\mu+\lambda)(\Div \vu)^2\big) \dx\dt \\
\leq & \int_\Omega \Big(\frac{|\vc{m}_0|^2}{2(\vr_0 +Z_0)} + H_P(\vr_0,Z_0)\Big) \dx
\end{aligned}
\end{equation}
for a.a. $\tau \in (0,T)$.
\end{df}

The first main result of the paper deals with system (\ref{eq1.1}--\ref{eq1.3}) and  reads
\begin{thm} \label{t1}
Let $\gamma \geq \frac 95$, $0<\beta<\infty$. 
Then under Hypotheses (H1--H5) problem (\ref{eq1.1}--\ref{eq1.3}) admits at least one weak solution in the sense of Definition \ref{d1}. 
Moreover, {for all $t\in \overline I$, $(\vr(t,x), Z(t,x))\in {\cal O}_{\underline a}$ for a. a. $x\in \Omega$},
$\vr \in C_{weak}([0,T); L^\gamma (\Omega)){\cap C(\overline I; L^1(\Omega))}$, $Z \in C_{weak}([0,T); L^{q_{\gamma,\beta}} (\Omega)){\cap C(\overline I; L^1(\Omega))}$, $(\vr+Z)\vu \in C_{weak}([0,T); L^q (\Omega;R^3))$ for some $q>1$ and $P(\vr,Z) \in {L^q(I\times\Omega)}$ for some $q>1$. 
In the above
$$
q_{\gamma,\beta}=\gamma\;\mbox{if $\beta < \gamma$},\;
q_{\gamma,\beta}=\beta\;\mbox{if $\beta\ge\gamma$}.
$$
\end{thm}

We are now in position to define weak solutions to problem (\ref{eq1.1bi}--\ref{eq1.3bi}).

\begin{df}\label{d1bi}
The quadruple $(\mathfrak{a}, \vr_-,\vr_+,\vu)$ is a bounded energy weak solution to problem (\ref{eq1.1bi}--\ref{eq1.3bi}), if $0\le\mathfrak{a}\le 1$,
$\vr_\pm \geq 0$ a.a. in $I\times \Omega$, $\vr^\pm \in L^\infty(I;L^{1}(\Omega))$, $\vu \in L^2(I;W^{1,2}_0(\Omega;R^3))$, $(\mathfrak{a}\vr_++(1-\mathfrak{a})\vr_- )|\vu|^2 \in L^\infty(I;L^1(\Omega))$, $\mathfrak{P}_-(\vr_-)= \mathfrak{P}_+(\vr_+) \in L^1(I\times \Omega)$, and: 
\begin{itemize}
\item Continuity equations (\ref{eq2.7}) are satisfied with $\vr=\mathfrak{a}\vr_+$ and $Z=(1-\mathfrak{a})\vr_-$;
\item Momentum equation (\ref{eq2.8}) is satisfied with $\vr=\mathfrak{a}\vr_+$, $Z=(1-\mathfrak{a})\vr_-$ and with function $P(\vr,Z)$ replaced by $\mathfrak{P}_+(\vr_+)$;
\item There is a non negative function $\mathfrak{H}:(0,1)\times(0,\infty)^2$ such that
\begin{equation} \label{eq2.9bi}
\int_\Omega \Big(\frac 12 (\vr+Z)|\vu|^2 + \mathfrak{H}(\mathfrak{a},\vr_-,\vr_+)\Big)(\tau,\cdot) \dx 
\end{equation}
$$
+ \int_0^\tau \int_\Omega \big(\mu|\Grad \vu|^2 + (\mu+\lambda)(\Div \vu)^2\big)\, {\rm d}x\, {\rm d}t
$$
$$
\leq \int_\Omega \Big(\frac{|\vc{M}_0|^2}{2(\vr_0 +Z_0)} + \mathfrak{H}(\mathfrak{a}_0,\vr_{-,0}, \vr_{+,0})\Big) \dx
$$
for a.a. $\tau \in (0,T)$.
\end{itemize}
\end{df}

The second main result of the paper deals with system (\ref{eq1.1bi}--\ref{eq1.3bi})
and reads:

\begin{thm} \label{t1bi}
Let $0\le\underline a<\overline a<\infty$. Let $G :=\gamma^+ + \gamma^+_{BOG}$
if $\underline{a} =0$ and $G:= \max\{\gamma^+ +\gamma^+_{BOG}, \gamma^- +\gamma^-_{BOG}\}$ if $\underline{a} >0$. Assume
\begin{equation}\label{gamma}
0<\gamma^-<\infty,\; \gamma^+\ge \frac{9}{5},\; \overline\Gamma< G,
\end{equation}
where
$$
\overline\Gamma=
\left\{\begin{array}{c}
\max\{\gamma^+ -\frac{\gamma^+} {\gamma^-}+ 1,
\, \gamma^- +\frac{\gamma^-} {\gamma^+}  -\frac{\gamma^+} {\gamma^-}\}\;\mbox{if $\underline a=0$}\\
\max\{\gamma^+ -\frac{\gamma^+} {\gamma^-}+ 1, 
\gamma^-
+\frac{\gamma^-} {\gamma^+} -1
 \}\;\mbox{if $\underline a>0$}
\end{array}
\right\}.
$$

Suppose that
\begin{equation}\label{Hbi1}
0\le\mathfrak{a}_0\le 1,\;\underline a\mathfrak{a}_0
\vr_{+,0}\le(1-\mathfrak{a}_0)\vr_{-,0}\le \overline a\mathfrak{a}_0
\vr_{+,0},
\end{equation}
$$
\vr_{+,0}\in L^{\gamma^+}(\Omega),\; {\frac{|\vc M_0|^2}{\mathfrak{a}_0
\vr_{+,0}+(1-\mathfrak{a}_0)\vr_{-,0}}\in L^1(\Omega).}
$$
Assume further that 
\begin{equation}\label{Hbi2}
\mathfrak{P}_{\pm}\in C([0,\infty))\cap C^{2}((0,\infty)),\;
\mathfrak{P}_{\pm}(0)=0,\;\mathfrak{P}_{\pm}'(s)>0, s>0,
\end{equation}
$$
\underline a_- s^{\gamma^- -1}- b_-\le \mathfrak{P}_{-}'(s), 
\; \mathfrak{P}_{-}(s)\le
\overline a_- s^{\gamma^-}+b_-,
$$
$$
\underline a_+ s^{\gamma^+ -1}- b_+\le \mathfrak{P}_{+}'(s)\; \le
\;
\overline a_+ s^{\gamma^+-1}+b_+,
$$
$$
|\mathfrak{P}_{\pm}''(s)| \leq d_{\pm} s^{A_{\pm}} + e_{\pm}, \: s\geq r>0,\; \mathfrak{P}_{+}''(s)\ge 0,\,s\in (0,\infty)
$$
with some positive constants $a_\pm$, $b_\pm$, $d_{\pm}$, $e_{\pm}$ and $A_{\pm}$.\footnote{Note that $d_{\pm}$ or $e_{\pm}$ may blow up when $r\to 0^+$.} Suppose further that
\begin{description}
\item{\it 1.}
\begin{equation}\label{Hbi3+}
\sup_{s\in (0,1)}s^{\underline\Gamma}\,\frac{\mathfrak P_{+}'(s+\mathfrak{q}^{-1}(\overline a s))(s+\mathfrak{q}^{-1}(\overline a s))^{2}}{s \mathfrak{q}(s)}\le\overline c<\infty
\end{equation}
with some $\underline\Gamma\in [0,1)$, 
where
\begin{equation}\label{mathfrakQ}
\mathfrak{q}=\mathfrak{P}_-^{-1}\circ\mathfrak{P}_+.
\end{equation}
\item{\it 2.}
\begin{equation}\label{Hbi4}
0<\underline q=\inf_{s\in (0,\infty)}\frac {\mathfrak{q}(s)}{s\mathfrak{ q}'(s) +\mathfrak{q}(s)}.
\end{equation}
\end{description}

Then  problem (\ref{eq1.1bi}--\ref{eq1.3bi}) admits at least one weak solution in the sense of Definition \ref{d1bi}. Moreover,  $\mathfrak{a}\vr_+$ belongs to the space $C_{weak}([0,T); L^{\gamma^+} (\Omega))$ and
and $(1- \mathfrak{a})\vr_-$ belongs to the space $C_{weak}([0,T); L^{q_{\gamma^+,\gamma^-}} (\Omega))$,
the vector field $(\mathfrak{a}\vr_+ + (1- \mathfrak{a})\vr_-)\vu$ belongs to $C_{weak}([0,T); L^r$ $ (\Omega;R^3))$ for some $r>1$, $\mathfrak{P}_{\pm}(\vr_\pm) \in L^r((0,T)\times\Omega)$ for some $r>1$ and
the function $\mathfrak{H}$ in the energy inequality (\ref{eq2.9bi}) is  given by formula
$$
\mathfrak{H}(\mathfrak{a},\vr_-,\vr_+)=
\mathfrak{a}\vr_+\int_0^{\mathfrak{a}\vr_+}\frac{\mathfrak{P}_+\circ \mathfrak{R}\Big(s,\,\frac{(1-\mathfrak{a})\vr_-}{\mathfrak{a}\vr_+}\,s\Big)}{s^2}\,{\rm d}s,
$$
where $\mathfrak{R}(s,z)$ is the unique solution in $[\mathfrak{a}\vr_+,\infty)$ of equation
$$
\mathfrak{R} \mathfrak{q}(\mathfrak{R})-\mathfrak{q}(\mathfrak{R})s-\mathfrak{R}z=0.
$$
\end{thm}

\noindent
\noindent
{\bf Remark \ref{se2}.3}
\noindent
{\it
\begin{description}
\item{\it 1.} It is to be noticed that the results of both Theorems \ref{t1} and \ref{t1bi}
remain valid --- after well known modifications in the definition of weak solutions in these cases --- in the space periodic setting (if $\Omega$ is a periodic periodic cell) or if we replace the no-slip boundary conditions
 (\ref{eq1.2}), and (\ref{eq1.2bi}), respectively, by the Navier conditions
$$
\vu\cdot\vc n|_{\partial\Omega}=0, \;\Big[\mu\Big(\nabla\vu+\nabla\vu^T-\frac 23\tn{I}{\rm div}\vu)+(\lambda-\frac\mu 3) \tn{I}{\rm div}\vu\Big]\vc n\times\vc n|_{\partial\Omega}={\bf 0}.
$$
In the above $\tn I$ denotes the identity tensor on $R^3$ and $\vc n$ is the outer normal to $\partial\Omega$.
\item{\it 2.} Condition (\ref{Om}) on the regularity of the domain $\Omega$ in both Theorems  \ref{t1} and \ref{t1bi}  could be relaxed up to a bounded Lipschitz domain via the technique described in \cite{FNPdom}. 
\item{\it 3.} In the case $0<\beta<{\gamma+\gamma_{BOG}}$ assumption (\ref{eq2.2}) of Theorem \ref{t1} can be weaken as
$$
\underline C (\vr^\gamma -1)\le P(\vr,Z) \leq \overline C(\vr^\gamma + Z^\beta +1),
$$
 for all $(\vr,Z)\in {\cal O}_{\underline a}$ at expense of slightly stronger assumption {(\ref{eq2.5-})} which has to be valid in this case with
$0<\overline\Gamma<\max\{\gamma+\gamma_{BOG}, \gamma+\beta_{BOG}\}$ if $\underline a>0$. Under this circumstances the value of $q_{\gamma,\beta}$ in Theorem 1
is $\gamma$ if $\beta<\gamma+\gamma_{BOG}$ and $\beta$ if $\gamma\ge \gamma+\gamma_{BOG}$.
\item{\it 4.}  Assumptions of Theorem \ref{t1bi} imposed on partial pressure laws are numerous but not so much restrictive: Suppose that If $\mathfrak{P}_\pm(\vr)\sim_0
\vr^{\alpha^\pm}$ and $\mathfrak{P}_\pm(\vr)\sim_\infty
\vr^{\gamma^\pm}$. Then all hypotheses are satisfied provided $\alpha^+>0$, $\alpha^->\frac{\alpha^+}{\sqrt{\alpha^+ +1}}$,
$\gamma^+\ge 9/5$ and, if $\underline{a}=0$, $0<\gamma^-<\frac{3\gamma^+}{6-2\gamma^+}$ for $\gamma^+<3$, and $\gamma^- {+\frac{(\gamma^-)^2 -(\gamma^+)^2}{\gamma^- \gamma^+}}<G$; if $\underline{a}>0$, then $\gamma^->0$ arbitrary. 
\item{\it 5.} Note that condition (\ref{Hbi3+}) in case $\mathfrak{P}'_+(0) >0$ can be formulated as $ \mathfrak q(s) \geq C s^A$, $A<2$, $s \in (0,1)$.
This can be shown directly from formula \eqref{it3} from the proof of Theorem \ref{t1bi}. 
\end{description}
}

The plan of the paper is following. Section \ref{se4} contains several preliminaries which are useful throughout the existence proof. We particularly concentrate
to the work with continuity and transport equations, which lies in the essence of the proofs.
Section \ref{se2+} is devoted to the proof of Theorem \ref{t1}. It is divided to several parts:
In Subsection \ref{se3} we introduce the approximate system for the model bi-fluid problem (\ref{eq1.1}--\ref{eq1.3}) which depends on three parameters: $N\in \N$, the dimension of the Galerkin approximation for the velocity, $\varepsilon>0$, the parabolic regularization of the continuity equations, and $\delta>0$, the regularization of the pressure.  Subsection \ref{se5} then deals with the existence of a solution to the third level of approximation, the Galerkin approximation for the velocity, and presents the main ideas of the limit passage $N \to \infty$. Subsection \ref{se6} deals with the limit passage $\varepsilon \to 0$ and Subsection \ref{se7} with the limit passage $\delta \to 0$. After completing the
investigation of the model bi-fluid system we will come back to the real bi-fluid system (\ref{eq1.1bi}--\ref{eq1.3bi}) and prove Theorem \ref{t1bi}. This will be done in Section \ref{sdod1}: we will show how to transform the real bi-fluid problem
to the model bi-fluid problem and complete the proof by verifying that the new pressure function in the transformed system verifies all Hypotheses (H3--H5).
 The last section contains discussion of the extension of the results for the model multiphase problem.

\section{Preliminaries} \label{se4}

This section contains several preliminary results used later. 

\subsection{Renormalized continuity and transport equations}

The properties of renormalized solutions to the continuity and transport equations
play important role in this paper. In this section, we shall recall those
needed for our method of the proofs (and prove some of those, whose proofs are not available in the needed form in the literature). The common denominator of all these results is the
DiPerna--Lions renormalization and regularization techniques for the transport equation introduced in \cite{DL}.

A pair of functions $(\vr,\vu)\in L^1(Q_T)\times L^1(0,T; W^{1,1}(\Omega))$ is called a renormalized solution to the continuity equation if
it satisfies
\begin{equation}\label{ce1}
\partial_t\vr+{\rm div}(\vr\vu)=0 \;\mbox{in ${\cal D}'(Q_T)$}
\end{equation}
and
\begin{equation}\label{ce2}
\partial_tb(\vr)+{\rm div}(b(\vr)\vu)+(\vr b'(\vr)- b(\vr)){\rm div}\vu=0 \;\mbox{in ${\cal D}'(Q_T)$}
\end{equation}
with any
\begin{equation}\label{clb}
b\in C^1([0,\infty)),\; b'\in L^\infty((0,\infty)).
\end{equation}

Likewise, a pair of functions $(s,\vu)\in L^1(Q_T)\times L^1(0,T; W^{1,1}(\Omega))$ is called a renormalized solution to the transport equation if
it satisfies
\begin{equation}\label{tr1}
\partial_t s +\vu\cdot\nabla s=0\;\mbox{in ${\cal D}'(Q_T)$},
\end{equation}
and
\begin{equation}\label{tr2}
\partial_t b(s) +\vu\cdot\nabla b(s)=0\;\mbox{in ${\cal D}'(Q_T)$}.
\end{equation}

The following proposition resumes the classical consequences of the Di\-Per\-na--Lions transport theory \cite{DL} applied to the continuity equation, as formulated and resumed in \cite[Chapter 4]{EF70}, \cite[Chapter 7]{NoSt}, \cite[Chapter 10]{FeNoB} We will suppose here that $\Omega$ is a bounded Lipschitz domain (even if some of statements do not need this assumption):

\begin{prop} \label{p1}
\begin{description}
\item{\it 1.} Let
\begin{equation}\label{ce3}
\vr\ge 0, \;\vr\in L^2(Q_T),\;\vu\in L^2(I;W^{1,2}(\Omega;R^3))
\end{equation}
satisfy continuity equation (\ref{ce1}). Then it satisfies its renormalized form (\ref{ce2}).
\item{\it 2.}
Let in addition to (\ref{ce3}),
\begin{equation}\label{ce4}
\vr\ge 0, \;\vr\in L^\infty(I;L^\gamma(\Omega)),\; \gamma >1.
\end{equation}
Then
$$
\vr\in C(\overline I;L^1(\Omega)).
$$
\item {\it 3.} If in addition to (\ref{ce3}), (\ref{ce4})
\begin{equation}\label{ce5}
\vu\in L^2( I; W^{1,2}_0(\Omega;R^3)),
\end{equation}
then both (\ref{ce1}--\ref{ce2}) hold up to the boundary and  in time integrated form,
namely: 1)
\begin{equation}\label{ce6}
\int_\Omega(\vr\varphi)(\tau,\cdot)\,{\rm d} x-\int_\Omega(\vr\varphi)(0,\cdot)\,{\rm d} x=\int_0^\tau\int_{\Omega}\Big(\vr\partial_t\varphi+\vr\vu\cdot\nabla\varphi\Big)\,{\rm d} x\,{\rm d}t
\end{equation}
for all $\tau\in\overline I$ and $\varphi\in C^1(\overline{Q_T})$. \\
2) For any $b$ in class (\ref{clb})
\begin{equation}\label{c7-}
b(\vr)\in C(\overline I; L^1(\Omega)),
\end{equation}
and
\begin{equation}\label{ce7}
\int_\Omega \big(b(\vr)\varphi\big)(\tau,\cdot)\,{\rm d} x-\int_\Omega \big(b(\vr)\varphi\big)(0,\cdot)\,{\rm d} x
\end{equation}
$$
=\int_0^\tau\int_{\Omega}\Big(b(\vr)\partial_t\varphi+b(\vr)\vu\cdot\nabla\varphi-(\vr b'(\vr)-b(\vr)){\rm div}\,\vu\,\varphi\Big)\,{\rm d} x\,{\rm d}t
$$
for all $\tau\in\overline I$ and $\varphi\in C^1(\overline{Q_T})$. 
\item{\it 4.} If in item 3. $b$ belongs solely to
\begin{equation}\label{clb+}
b\in C([0,\infty))\cap C^1((0,\infty)), \; b(s)\le c(1+s^{\frac 56\gamma}),
\end{equation}
$$
\;
sb'-b\in C([0,\infty)),\; sb'(s)-b(s)\le c(1+ s^{\gamma/2}),
$$
then
\begin{equation}\label{c7-b}
b(\vr)\in C_{\rm weak}(\overline I; L^p(\Omega)),\;\mbox{with any $1\le p<6/5$}
\end{equation}
and equation (\ref{ce7}) still continues to hold.
\end{description}
\end{prop}
\noindent
{\bf Remark \ref{se4}.1}
{\it We notice that function  $b(\vr)=\vr\ln\vr$ as well as its truncation
\begin{equation}\label{Lk}
L_k(\vr)=\vr\int_1^\vr\frac{T_k(z)}{z^2}\,{\rm d}z,\; T_k(z)=k T(Z/k), \; k>1,
\end{equation}
$$
T(z)=\left\{\begin{array}{c}
z\mbox{ if } z\in [0,1)\\
2\mbox{ if } z\ge 3
\end{array}
\right\},\; T\in C^{\infty}([0,\infty)),\;\mbox{concave}
$$
verify  assumptions of item 4. of Proposition \ref{p1}. It is this function which is employed in the Feireisl--Lions approach at the very last step of the proof of compactness of the density sequence.}
\\ \\
{\bf Remark \ref{se4}.2}
{\it The statement of Proposition \ref{p1} holds also if the couple $(s,\vu)$
satisfies transport equation (\ref{tr1}) on place of the continuity equation.
In this case: 1) Equation (\ref{ce6}) takes form
\begin{equation}\label{tr6}
\int_\Omega (s\varphi)(\tau,\cdot)\,{\rm d} x-\int_\Omega (s\varphi)(0,\cdot\,){\rm d} x
\end{equation}
$$
=\int_0^\tau\int_{\Omega}\Big(s\partial_t\varphi+s\vu\cdot\nabla\varphi
-s{\rm div}\,\vu\,\varphi\Big)\,{\rm d} x\,{\rm d}t
$$
for all $\tau\in\overline I$ and $\varphi\in C^1(\overline{Q_T})$. 2) For any $b$ in class (\ref{clb})
\begin{equation}\label{c7-!}
b(s)\in C(\overline I; L^1(\Omega)),
\end{equation}
and
\begin{equation}\label{tr7}
\int_\Omega \big(b(s)\varphi\big)(\tau,\cdot){\rm d} x-\int_\Omega \big(b(s)\varphi\big)(0,\cdot){\rm d} x
\end{equation}
$$
=\int_0^\tau\int_{\Omega}\Big(b(s)\partial_t\varphi+b(s)\vu\cdot\nabla\varphi-b(s) {\rm div}\,\vu\,\varphi\Big)\,{\rm d} x\,{\rm d}t
$$
for all $\tau\in\overline I$ and $\varphi\in C^1(\overline{Q_T}).$ 

Equation (\ref{tr7}) continues to hold even with
$$
b\in C([0,\infty))\cap C^1((0,\infty)),\;b(s)\le c(1+s^{\frac \gamma 2}),\;\gamma>1.
$$
In this case, however, $b(\vr)$ belongs solely to $C_{\rm weak}(\overline I, L^p(\Omega))$ for all $1\le p<2$.
} 
\\ \\
The next proposition is a generalization of Proposition \ref{p1} to several
continuity equations and renormalizing functions of several variables. Such generalizations appear to be useful for applications when treating models of compressible fluids with
pressure dependent on several variables subject to continuity equations. One of the first attempts in this directions for particular renormalizing functions were
formulated in (\cite[Lemma 6.1]{FKNZ}, \cite[Section 8.1]{3MNPZ} and studied more systematically later in Vasseur et al. \cite[Lemma 2.4]{VWY}.

\begin{prop} \label{p2}
\begin{description}
\item{\it 1.}
Let   couples $(\vr,\vu)$, $(Z,\vu)$
$$
\vr \in L^2(Q_T), \;(\vr, Z)\in {\cal O}_{0},\;\vu\in L^2(I;W^{1,2}(\Omega;R^3)),
$$
(cf. (\ref{eq2.1})) verify continuity equation (\ref{ce1}). Then for any 
$$
b\in C^1([0,\infty)^2), (\partial_{\vr} b,\partial_Z b)\in L^\infty({\cal O}_0;R^2)
$$
 the function $b(\vr, Z)$
verifies the renormalized continuity equation
\begin{equation}\label{ce2+}
\partial_tb(\vr,Z)+{\rm div}(b(\vr,Z)\vu)
\end{equation}
$$
+(\partial_\vr b(\vr,Z)- Z \partial_Z b(\vr,Z)-b(\vr,Z)){\rm div}\vu=0 \;\mbox{in ${\cal D}'(Q_T)$.}
$$
\item{\it 2.}
If moreover  we have $\vr\in L^\infty(I;L^\gamma(\Omega))$ with some $\gamma>1$ and $\vu\in L^2(I,W_0^{1,2}(\Omega;R^3))$, then
$$
\vr, Z\in C(\overline I; L^1(\Omega)),\; b(\vr,Z) \in C(\overline I; L^1(\Omega))
$$
and equation (\ref{ce2+}) holds in the time integrated form up to the boundary:
\begin{equation}\label{ce7+}
\int_\Omega \big(b(\vr,Z)\varphi\big)(\tau,\cdot)\,{\rm d} x-\int_\Omega \big(b(\vr, Z)\varphi\big)(0,x)\,{\rm d} x =\int_0^\tau\int_{\Omega}\Big(b(\vr,Z)\partial_t\varphi
\end{equation}
$$
+b(\vr,Z)\vu\cdot\nabla\varphi-(\vr \partial_\vr b(\vr,Z)- Z \partial_Z b(\vr,Z)-b(\vr,Z)){\rm div}\,\vu\, \varphi\Big)\,{\rm d} x\,{\rm d}t
$$
for all $\tau\in\overline I$ and $\varphi\in C^1(\overline{Q_T})$.

\end{description}
\end{prop}

We would need to apply this proposition to function $b(\vr,Z)=Z/\vr$ which however
does not satisfy the requested hypotheses. In order to circumvent this difficulty, we shall prove the following proposition:

\begin{prop}\label{bsing}
Let $\vr,Z$ belong to $C(\overline I;L^1(\Omega))$. We define {\it for all} $t\in \overline I$,
 $s(t,x)=\frac{Z(t,x)}{\vr(t,x)}$ in agreement with (\ref{conv}).
Then we have:
\begin{description}
\item{\it 1.} 
If {\it for all $t\in \overline I$}, $0\le Z(t,\cdot)\le \overline a \vr(t,\cdot)$ a.e. in $\Omega$, then
\begin{equation}\label{s2}
\mbox{for all $t\in \overline I$, $0\le s(t,\cdot)\le \overline a$ a.e. in $\Omega$}.
\end{equation}
\item{\it 2.} Suppose moreover that
$$ 
\vr\in L^2(Q_T)\cap L^\infty(I;L^\gamma(\Omega)),\;\mbox{with some $\gamma>1$}
$$
and that both couples $(\vr,\vu)$ and $(Z,\vu)$ satisfy continuity equation (\ref{ce1}) with $\vu\in L^2(I;W^{1,2}(\Omega;R^3))$. Then 
\begin{equation}\label{s3}
s\in C(\overline I;L^q(\Omega))\;\mbox{for all $1\le q<\infty$}
\end{equation}
and the couple $(s,\vu)$ satisfies transport equation (\ref{tr1}).
\item{\it 3.} If moreover $\vu\in L^2(I;W_0^{1,2}(\Omega;R^3))$, then transport equation  (\ref{tr1}) holds up to the boundary in the time integrated form (\ref{tr6}).
\end{description}
\end{prop}

\begin{proof}
The first statement (item 1.) is a direct consequence of the definition of function  $s$.

Proof of the second statement (item 2.): We consider sequence of functions
$$
b_\delta: [0,\infty)^2\mapsto[0,\infty), \; b_\delta(\vr,Z)= \frac Z{\vr+\delta},\;
\delta\in (0,1/2).
$$
Functions $b_\delta$ and couples $(\vr,\vu)$ and $(Z,\vu)$ satisfy all assumptions of Proposition \ref{p2}. Consequently,
$b_\delta(\vr, Z)\in C(\overline I, L^q(\Omega))$, $1\le q<\infty$ (by interpolation)
and
\begin{equation}\label{ce7+.1-}
\partial_t b_\delta(\vr,Z)+{\rm div}(b_\delta(\vr,Z)\vu) -\frac{Z\vr}{(\vr+\delta)^2}
{\rm div}\vu=0\;\mbox{in ${\cal D}'(Q_T)$}
\end{equation}
if $\vu\in L^2(I;W^{1,2}(\Omega;R^3))$, or even
\begin{equation}\label{ce7+.1}
\int_\Omega \big(b_\delta(\vr,Z)\varphi\big)(\tau,\cdot)\,{\rm d} x-\int_\Omega \big(b_\delta(\vr, Z)\varphi\big)(0,\cdot)\,{\rm d} x =\int_0^\tau\int_{\Omega}\Big(b_\delta(\vr,Z)\partial_t\varphi
\end{equation}
$$
+b_\delta(\vr,Z)\vu\cdot\nabla\varphi+\frac{Z\vr}{(\vr+\delta)^2}
\,{\rm div}\,\vu \, \varphi\Big)\,{\rm d} x\,{\rm d}t
$$
if $\vu\in L^2(I;W_0^{1,2}(\Omega;R^3))$.

We easily verify that
$$
\mbox{for all $t\in\overline I$,} |b_\delta(\vr,Z)(t,\cdot)|+\Big|\frac{Z\vr}{(\vr+\delta)^2}(t,\cdot)
\Big|\le \overline a
\;\mbox{a.a. in $\Omega$},
$$
$$
\mbox{for all $t\in\overline I$,}\;b_\delta(\vr(t,\cdot), Z(t,\cdot))\to s(t,\cdot)\;\mbox{a.a. in $\Omega$ as $\delta\to 0$}.
$$
We thus obtain via the Lebesgue dominated convergence theorem
$$
\mbox{for all $\tau\in\overline I$,}\;
\int_\Omega (b_\delta(\vr,Z)\varphi)(\tau,\cdot)\,{\rm d} x\to \int_\Omega (s\varphi)(\tau,\cdot)\,{\rm d} x
$$
$$
\mbox{for all $\tau\in\overline I$,}\;\int_0^\tau\int_\Omega b_\delta(\vr,Z)\vu\cdot\nabla\varphi\,{\rm d}x\,{\rm d}t\to \int_0^\tau\int_\Omega s\vu\cdot\nabla\varphi\,{\rm d}x\,{\rm d}t
$$
and
$$
\mbox{for all $\tau\in\overline I$,}\;\int_0^\tau\int_\Omega \frac{Z\vr}{(\vr+\delta)^2}
{\rm div}\,\vu \varphi\,{\rm d} x\,{\rm d}t\to\int_0^\tau\int_\Omega s
{\rm div}\,\vu \varphi\,{\rm d} x\,{\rm d}t.
$$
We can therefore pass to the limit in equations (\ref{ce7+.1-}), (\ref{ce7+.1})
in order to recover equations (\ref{tr1}) or (\ref{tr6}), according to the case.

Finally, according to Remark \ref{se4}.2, function $s$ belongs to $C(\overline I;L^1(\Omega))$, and a fortiori, by interpolation, also to  $C(\overline I;L^q(\Omega))$,
$1\le q<\infty$. This completes the proof of Proposition \ref{bsing}. 
\end{proof}
 
We shall still need to combine solutions of continuity equation with solutions of transport equation. The next statement is again a straightforward consequence of the
DiPerna--Lions transport theory.

\begin{prop} \label{p2+}
\begin{description}
\item{\it 1.}
Let
$$
\vr \in L^2(Q_T), s\in L^\infty(Q_T),\;\vu\in L^2(I;W^{1,2}(\Omega;R^3)),
$$
and let the couple $(\vr,\vu)$ verify continuity equation (\ref{ce1}) and the couple
$(s,\vu)$ transport equation (\ref{tr1}).
Then $s\in C(\overline I;L^1(\Omega))$ and the product $s\vr$ satisfies the continuity equation in the sense of distributions
on $Q_T$. 
\item{\it 2.}
If moreover we have $\vr\in L^\infty(I;L^\gamma(\Omega))$ with some $\gamma>1$ and $\vu\in L^2(I,W_0^{1,2}(\Omega;R^3))$, then
$$
\vr\in C(\overline I; L^1(\Omega)),\; s \vr \in C(\overline I; L^1(\Omega))
$$
and the continuity equation  for $s\vr$ holds in the time integrated form up to the boundary:
\begin{equation}\label{ctr}
\int_\Omega (s\vr\varphi)(\tau,\cdot)\,{\rm d} x-\int_\Omega (s\vr\varphi)(0,\cdot)\, {\rm d} x =\int_0^\tau\int_{\Omega}\Big(s\vr\partial_t\varphi
+ s\vr\vu\cdot\nabla\varphi\Big)\,{\rm d} x\,{\rm d}t
\end{equation}
for all $\tau\in\overline I$ and $\varphi\in C^1(\overline{Q_T})$.
\end{description}
\end{prop}

The next proposition is one of the crucial point of the compactness argument, compare with Vasseur et al. \cite{VWY}.
\begin{prop}\label{L2}
\begin{description}
\item{\it 1.} 
Let 
$$
\vu_n\in L^2(I,W_0^{1,2}(\Omega;R^3)),\;
(\vr_n, Z_n)\in {\cal O}_0\cap \Big(C(\overline I;L^1(\Omega))\cap L^2(Q_T)\Big)^2.
$$
Suppose that
\begin{multline*}
\sup_{n\in N}\Big(\|\vr_n\|_{L^\infty(I;L^\gamma(\Omega))}+\|Z_n\|_{L^\infty
(I;L^{\gamma}(\Omega))}\\ + \|\vr_n\|_{L^2(Q_T)}+ \|\vu_n\|_{L^2(I;W^{1,2}(\Omega))}\Big)
<\infty,
\end{multline*}
where $\gamma>6/5$, 
and that both couples $(\vr_n,\vu_n)$, $(Z_n,\vu_n)$ satisfy continuity equation
(\ref{ce1}). Then, 
up to a subsequence (not relabeled)
$$
\begin{aligned}
(\vr_n, Z_n)&\to (\vr,Z)\;\mbox{in $(C_{\rm weak}(\overline I;L^\gamma(\Omega)))^2$},
\\
\vu_n&\rightharpoonup\vu\;\mbox{in $L^2(I;W^{1,2}(\Omega;R^3))$,}
\end{aligned}
$$
where $(\vr,Z)$ belongs to spaces
$$
{\cal O}_0\cap (L^2(Q_T))^2\cap (L^\infty(I,L^\gamma(I,\Omega)))^2\cap (C(\overline I;L^1(\Omega))^2
$$
and $(\vr,\vu)$ as well as (Z,\vu) verify continuity equation in the renormalized sense.
\item{\it 2.} We define in agreement with convention (\ref{conv}) for all $t\in \overline I$, 
\begin{equation}\label{sn}
s_n(t,x)=\frac{Z_n(t,x)}{\vr_n(t,x)},\quad s(t,x)=\frac {Z(t,x)}{\vr(t,x)}.
\end{equation}
Suppose in addition to assumptions of item 1. that
$$
\int_{\Omega}{\vr_n(0,x)}s_n^2(0,x){\rm d} x\to \int_{\Omega}\vr(0,x) s^2(0,x)\,{\rm d} x.
$$
Then $s_n, s\in C(\overline I;L^q(\Omega))$, $1\le q<\infty$ and  for all $t\in \overline I$, 
$0\le s_n(t,x)\le\overline a$, $0\le s(t,x)\le \overline a$ for a.a. $x\in \Omega$. Moreover, both
$(s_n,\vu_n)$ and $(s,\vu)$ satisfy
transport equation (\ref{tr6}). 
\item{\it 3.}
Finally,
\begin{equation}\label{cvs}
\int_{\Omega}(\vr_n|s_n-s|^p)(\tau,\cdot)\,{\rm d} x\to 0\;\mbox{with any $1\le p<\infty$}
\end{equation}
for all $\tau\in \overline I$.
\end{description}
\end{prop} 
\begin{proof}
The first item is nowadays mathematical folklore. We refer e.g. to monographs \cite{EF70}
or \cite{NoSt} and skip the proof.

The second item is nothing but the statement of Proposition \ref{bsing}.

We shall now show the third statement. We apply to sequence $s_n$ Proposition \ref{bsing}, notably the fact that it satisfies the transport equation (\ref{tr1}). We now employ conclusion of Remark \ref{se4}.2 and take advantage of the fact that
$s_n(t)$ admits for all $t\in\overline I$ an $L^\infty(\Omega)$ bound in order
to verify that couples $(s_n^2,\vu_n)$ satisfy  transport equation (\ref{tr1})$_{s=s_n^2,\vu=\vu_n}$ as well. Next we recall that couples $(\vr_n,\vu_n)$ verify continuity equation (\ref{ce1})$_{\vr_n,\vu_n}$; whence, in virtue of Proposition \ref{p2+}, functions $\vr_n s_n^2\in C(\overline I;L^1(\Omega))$ (and even in $C(\overline I;L^q(\Omega))$, $1\le q<\gamma,$ by interpolation), and, moreover, couples $(\vr_n s_n^2,\vu_n)$ satisfy time integrated continuity equation up to the boundary, see (\ref{ce6})$_{\vr=\vr_ns_n^2,\vu=\vu_n}$. Taking into account the convergence of initial data we get from  this equation with test function $\varphi=1$,
\begin{equation}\label{s1!!}
\lim_{n\to\infty}\int_{\Omega}({\vr_n} s_n^2)(\tau,\cdot)\,{\rm d}x=\int_{\Omega}
(\vr s^2)(0,\cdot)\,{\rm d}x\;\mbox{for all $\tau\in \overline I$}.
\end{equation}

Repeating the same reasoning with limits $(\vr,Z,s,\vu)$, we get by the same token
\begin{equation}\label{s2!!}
\lim_{n\to\infty}\int_{\Omega}(\vr_n s^2)(\tau,\cdot)\,{\rm d}x=
\int_{\Omega}(\vr s^2)(\tau,\cdot)\,{\rm d}x=\int_{\Omega} (\vr s^2)(0,\cdot)\,{\rm d}x\;\mbox{for all $\tau\in \overline I$}.
\end{equation}
The first identity holds due to the  convergence of $\vr_n$ in $C_{\rm weak}(\overline I;L^\gamma(\Omega))$.

Finally we have
$$
\lim_{n\to\infty}\int_\Omega(\vr_n s_n s)(\tau,\cdot)\,{\rm d}x=
\lim_{n\to\infty}\int_\Omega (Z_n s)(\tau,\cdot)\,{\rm d}x
$$
$$
= 
\int_\Omega (Z s)(\tau,\cdot)\,{\rm d}x\int_\Omega (\vr s^2)(\tau,\cdot)\,{\rm d}x = \int_\Omega (\vr s^2)(0,\cdot)\,{\rm d}x
$$
for all $\tau\in\overline I$.
The first identity holds due to (\ref{sn}), the second one due to the  convergence of $Z_n$ in $C_{\rm weak}(\overline I;L^\gamma(\Omega))$, the third one due to (\ref{conv})
and the last one due to (\ref{s2!!}). Resuming this part, we get
\begin{equation}\label{s3!}
\lim_{n\to\infty}\int_\Omega(\vr_n s_n s)(\tau,\cdot)\,{\rm d}x=\int_\Omega (\vr s^2)(0,\cdot)\,{\rm d}x \;\mbox{for all $\tau\in \overline I$}.
\end{equation} 

Putting together (\ref{s1!!}--\ref{s3!}) gives

$$
\lim_{n\to\infty}\int_{\Omega}(\vr_n( s_n-s)^2)(\tau,\cdot)\,{\rm d}x =0 \;\mbox{for all $\tau\in \overline I$}.
$$
This yields (\ref{cvs}) by interpolation.
\end{proof}

\subsection{Other tools}

The next result follows from the maximal parabolic regularity theory and comparison principle applied to the regularized continuity equation (\ref{eq3.2}), cf. e.g. Denk, Hieber, Pr\"uss \cite[Theorem 2.1]{DHP} and \cite[Proposition 7.39]{NoSt}. 
\begin{prop}\label{parabolic}
Suppose that  
{ $\vr_0\in W^{1,2}(\Omega)$}, $\vu\in L^\infty(0,T; W^{1,\infty}(\Omega;R^3))$, { $\vu|_{(0,T)\times\partial\Omega}=\mathbf{0}$}. Then we have:
\begin{description}
\item{\it 1.}
The parabolic  problem (\ref{eq3.2}) admits a unique solution in the class
\begin{equation}\label{parabolicr}
\vr\in L^2(I;W^{2,2}(\Omega))\cap W^{1,2}(I;L^2(\Omega)).
\end{equation}
\item{\it 2.} If moreover $0<\underline\vr\le\vr_0\le\overline\vr<\infty$ a.a.
in $\Omega$, then there is $0<\underline c<\overline c<\infty$ dependent on
$\tau,\underline\vr, \overline\vr$ and $\|{\rm div}\,\vu\|_{L^1(I;L^\infty(\Omega))}$ 
such that
$$
\mbox{for all $\tau\in \overline I$},\; \underline c\le \vr(\tau, x)\le\overline c\;
\mbox{for a.a. $x\in \Omega$}.
$$
\end{description}
\end{prop}

\begin{prop} \label{p3}
Let $P,G$: $(I\times\Omega)\times [0,\infty)\in R$ be a couple of functions such that
for almost all $(t,x)\in I\times\Omega$,  $\vr\mapsto P(t,x,\vr)$ and $\vr\mapsto G(t,x,\vr)$
are both non decreasing and continuous on $[0,\infty)$. Assume that  $\vr_n\in
L^1(Q_T)$ is a sequence such that
$$
\left.\begin{array}{c}
P(\cdot,\vr_n(\cdot)) \rightharpoonup \overline{P(\cdot,\vr)}, \\
G(\cdot,\vr_n(\cdot)) \rightharpoonup \overline{G(\cdot,\vr)}, \\
P(\cdot,\vr_n(\cdot))G(\cdot,\vr_n(\cdot)) \rightharpoonup \overline{P(\cdot,\vr)G(\cdot,\vr)}
\end{array} \right\} \mbox{ in } L^1(I\times\Omega).
$$
Then
$$
\overline{P(\cdot,\vr)}\, \, \overline{G(\cdot,\vr)} \leq
\overline{P(\cdot,\vr)G(\cdot,\vr)}
$$
a.a. in $I\times \Omega$.
\end{prop}

%
\begin{proof}
The proof is a straightforward modification of a similar result from \cite[Theorem 10.19]{FeNoB}.
\end{proof}

Finally we report  a convenient version of the celebrated Div-Curl lemma, see \cite[Section 6]{EF70} or \cite[Theorem 10.27]{FeNoB}. To this end we denote
by ${\tn R}$ the Riesz transform as a pseudodifferential operator
with the Fourier symbol $\frac {\xi\otimes\xi}{|\xi|^2}$. The statement reads.

\begin{prop}\label{rieszcom} 
Let
\[
\vc{ V}_\ep \to \vc{ V} \ \mbox{weakly in}\ L^p(R^3; R^3),
\]
\[
\vc{ U}_\ep \to \vc{ U} \ \mbox{weakly in}\ L^q(R^3; R^3),
\]
where $ \frac{1}{p} + \frac{1}{q} = \frac{1}{r} < 1$. Then
\[
\vc{ U}_\ep \cdot ({\tn R}\cdot[\vc{V}_\ep ]) - ({\tn R}\cdot[\vc{
U}_\ep])\cdot \vc{ V}_\ep \to \vc{ U}\cdot ({\tn R}\cdot[\vc{ V}]) -
({\tn R}\cdot[\vc{ U}])\cdot \vc{ V} \ \mbox{weakly in}\ L^r(R^3).
\]
In the above, $({\tn R}\cdot[\vc{ U}])_i=\sum_{j=1}^3\tn R_{ij}[U_j].$
\end{prop}

\section{The model bi-fluid system: Proof of Theorem \ref{t1}}\label{se2+}
\subsection{Approximate system} \label{se3}

We introduce a three level approximation which is similar to the standard procedure applied for the existence proof to the evolutionary compressible Navier--Stokes equation as suggested in \cite{FNP} and modified in \cite{3MNPZ} in order to accommodate several species. Since we deal with generally non-monotone pressure, we combine this approach with the ideas from \cite{Fe2002}. Finally, other particular modifications of the standard approach are needed due to  the fact that we deal with two-densities pressure.

We may suppose,  without loss of generality, that
\begin{equation}\label{d0}
P\in C^2((0,\infty)^2)
\end{equation}
and that $\partial^2_Z P$ and $\partial_\vr P$ verify bound (\ref{eq2.3a-}).
If this is not the case, we would be obliged to replace, in what follows,
function $P$ by a regularized function $[P]_\xi$ such that
\begin{equation}\label{chi}
[P]_{\xi}(\vr,Z)=\int_{R^3}P(\vr-\sigma,Z-z)\chi_\xi(\sigma)\chi_\xi(z)\,{\rm d} \sigma\,{\rm d}z,
\end{equation}
where $\chi_\xi$, $\xi>0$ is the standard even one dimensional mollifier supported in $(-\xi,\xi)$. This would augment  by one the number of parameters in the approximate system and would require to insert to the construction process between the passage $\ep\to 0$ and before the passage $\delta\to 0$ the limit $\xi\to 0$, which would use the same argumentation as the passage $\delta\to 0$, cf. equations (\ref{eq3.2}--\ref{eq3.4}).

We first take $\delta >0$ and a sufficiently large $B\gg  {\rm max}\{9/2,\gamma,\beta, A\}$.  Further, we take $\eta_\delta(x) = \eta(x/\delta)$ a smooth cut-off function such that $\eta\in C^\infty([0,\infty))$,
\begin{equation}\label{eta}
\begin{aligned}
\eta(z) &= \left\{
\begin{array}{rl}
1 \quad &\text { for } 0\leq z\leq 1/2 \\
0 \quad &\text { for } 1<z \\
\in (0,1) &\text { for } 1/2<z<1
\end{array}
\right\},\\
0&\le -\eta'(z)\le 2\;\mbox{ for all $z$}.
\end{aligned}
\end{equation}
Then we define
\begin{equation} \label{eq3.1}
\Pi_\delta(\vr,Z) = P_\delta (\vr,Z) + \delta \Big(\vr^B + Z^B + \frac 12 \vr^2 Z^{B-2} + \frac 12 Z^2 \vr^{B-2}\Big),
\end{equation}
where
$$
P_\delta (\vr,Z)=\big(1-\eta_\delta(\sqrt{\vr^2+Z^2})\big) P(\vr, Z). 
$$
We notice, in particular, that the regularized $P_\delta$ conserves all Hypotheses
(H3--H5), namely (\ref{eq2.2}), (\ref{eq2.5-}), (\ref{eq2.4}) and a fortiori (\ref{eq2.3}), (\ref{eq2.5}) with $C$ still independent of $\delta$. The coefficient $C$ in assumption (\ref{eq2.3a-}) and a fortiori in assumption (\ref{eq2.3a}) may however explode as $\delta\to 0$. Furthermore, the function $(\vr,Z)\mapsto \vr^B + Z^B + \frac 12 \vr^2 Z^{B-2} + \frac 12 Z^2 \vr^{B-2}$ is convex in $[0,\infty)^2$. 

We can also suppose without loss of generality that initial conditions are regular enough with densities $(\vr_0, Z_0)$ out of vacuum, namely
\begin{equation}\label{dinit}
0<\vr_0<\overline a Z_0,\; (\vr_0,Z_0)\in C^3(\overline\Omega),\;(\partial_{\vc n}\vr_0,\partial_{\vc n} Z_0)|_{\partial\Omega}=(0,0)
\end{equation}
$$
\vu_0\in C^3(\overline\Omega;R^3)\cap W_0^{1,2}(\Omega;R^3).
$$
If this is not the case, and they verify only Hypotheses (H1--H2), we can mollify them using parameter $\delta$ to $( \vr_{0}^\delta, Z_0^\delta, \vu_0^\delta)$ belonging to (\ref{dinit}) in such a way that they converge to $( \vr_{0}, Z_0, \vu_0)$ in the spaces (\ref{eq2.6}) of initial conditions, as well as $(Z_0^\delta)^2/\vr_0^\delta\to
\vr_0s_0^2$, $(\vr_0^\delta+Z_0^\delta)|\vu_0^\delta|^2\to (\vr_0+Z_0)|\vu_0|^2$ in $L^1(\Omega)$, see \cite{3MNPZ} for more details.


Next step consists in the parabolic regularization of the continuity equations by adding $\varepsilon \Delta \vr$ and $\varepsilon \Delta Z$ to the right-hand side of the continuity equations for densities $\vr$, $Z$ and endowing the new equations  with the homogeneous Neumann boundary conditions. This regularization is compensated by adding the term $\varepsilon \Grad (\vr+Z) \cdot \Grad \vu$ to the left-hand side of the momentum equation in order to keep in force the energy identity.

Finally we take $\{\vcg{\Phi^j}\}_{j=1}^\infty\subset C^2(\overline\Omega;R^3))\cap W^{1,2}_0(\Omega;R^3)$ an orthonormal basis in $L^2(\Omega;R^3))$  (formed e.g. by eigenfunctions of the Lam\'e system with homogeneous Dirichlet boundary conditions) and consider for  a fixed $N \in \N$ an orthogonal projection of the momentum equation onto the linear hull ${\rm LIN}\{\vcg{\Phi^j}\}_{j=1}^N$. 

To summarize, our approximation looks as follows:

\begin{df} \label{d2}
The triple\footnote{We skip the indices $N$, $\varepsilon$ and $\delta$ in what follows and will use (only one of them) in situations when it will be useful to underline the corresponding limit passage.}  ($\vr^{N,\varepsilon, \delta}, Z^{N,\varepsilon, \delta}, \vu^{N,\varepsilon, \delta}) = (\vr,Z,\vu)$ is a solution to our approximate problem, provided $\partial_t \vr$, $\partial_t Z$, $\Grad^2 \vr$ and $\Grad^2 Z \in L^r(I\times \Omega)$ for some $r\in (1,\infty)$, $\vu(t,x) =\sum_{j=1}^N c_j^N(t) \vcg{\Phi}_j(x)$ with $c_j^N \in C^1(0,T) \cap C([0,T])$ for $j=1,2,\dotsm, N$, the regularized continuity equation problems
\begin{equation} \label{eq3.2}
\begin{aligned}
\partial_t \vr + \Div(\vr\vu) &= \varepsilon \Delta \vr \\
\pder{\vr}{\vc{n}}\Big|_{\partial \Omega} &= 0 \\
\vr(0,x) &= \vr_0, 
\end{aligned}
\end{equation} 
and  
\begin{equation} \label{eq3.3}
\begin{aligned}
\partial_t Z + \Div(Z\vu) &= \varepsilon \Delta Z \\
\pder{Z}{\vc{n}}\Big|_{\partial \Omega} &= 0 \\
Z(0,x) &= Z_0
\end{aligned}
\end{equation}	
hold in the a.a. sense, and the Galerkin approximation for the momentum equation
\begin{equation} \label{eq3.4}
\begin{aligned}
&\int_0^T \int_\Omega \Big(\partial_t \big((\vr+ Z)\vu\big)\vcg{\varphi} -(\vr+Z)(\vu\otimes\vu):\Grad \vcg{\varphi} - \Pi_\delta(\vr,Z)\Div \vcg{\varphi}\Big)\dx \dt \\
=& \int_0^T \int_\Omega \Big(\mu \Grad \vu:\Grad \vcg{\varphi} + (\mu+\lambda) \Div \vu \, \Div \vcg{\varphi} - \varepsilon \big(\Grad (\vr+Z)\cdot \Grad \vu\big)\cdot \vcg{\varphi}\Big) \dx\dt  
\end{aligned}
\end{equation}
holds for any $\vcg{\varphi} \in {\rm {LIN}}\{\vcg{\Phi}\}_{j=1}^N$, and 
\begin{equation} \label{eq3.5}
\vu(0,x) = {\mathcal P}_N(\vu_0)
\end{equation}
 with ${\mathcal P}_N$ the orthogonal projection onto ${\rm {LIN}}\{\vcg{\Phi}\}_{j=1}^N$ in $L^2(\Omega;R^3)$.
\end{df}


\subsection{Galerkin approximation} \label{se5}

We are now prepared to study existence of solution on the level of the Galerkin approximation, as specified in Definition \ref{d1}. We shall only prove a priori estimates independent of $N$, as the rest is more or less similar to the proof for the compressible (mono-fluid) Navier--Stokes equations. The local in time existence is proved by means of the local in time existence theory for ordinary differential equation (existence  for the Galerkin approximation for the velocity) combined with the existence theory for linear parabolic equations (continuity equations). These solutions are put together via a version of the Schauder fixed point theorem and due to estimates proved below this solution can be extended on the whole time interval $(0,T)$. The details can be found in \cite[Chapter 7]{NoSt}.

Let us concentrate now on the estimates independent of the parameter $N$. 

First, applying Proposition \ref{p3} to the both regularized continuity equations (\ref{eq3.2}--\ref{eq3.3})$_{\vr^N,Z^N,\vu^N}$ and to the regularized continuity equation satisfied by the  differences $\overline a\vr^N-Z^N$ and $Z^N-\underline{a} \vr^N$, we easily see that for all $t \in \overline{I}$ and $x \in \Omega$
\begin{equation} \label{eq5.1} 
\begin{aligned}
\vr^N(t,x)&\geq C_1(\delta, N)>0, \quad Z^N(t,x) \geq C_1(\delta, N)>0, \\
\underline a \vr^N(t,x) &\leq Z^N(t,x) \leq \overline a \vr^N(t,x) \leq C_2(\delta,N)<\infty.
\end{aligned}
\end{equation}

Next, we multiply equation \eqref{eq3.2}$_{\vr^N,\vu^N}$ by $\vr^N$, equation \eqref{eq3.3}$_{Z^N,\vu^N}$ by $Z^N$ and integrate over $\Omega$. We end up with
\begin{multline} \label{eq5.2}
\frac 12 \frac{{\rm d}}{{\rm d}t}\|\vr^N,Z^N\|^2_{L^2(\Omega)} + \varepsilon \|\Grad \vr^N,\Grad Z^N\|_{L^2(\Omega)}^2 \\ = -\frac 12 \int_\Omega \big((\vr^N)^2+ (Z^N)^2\big)\Div \vu^N\dx.
\end{multline} 
Finally, we use in the projected momentum equation (\ref{eq3.4})$_{\vr^N,Z^N\vu^N}$ as test function the solution $\vu^N$ itself, which is certainly possible. Using also the continuity equations we end up with
\begin{equation} \label{eq5.3}
\frac 12 \frac{{\rm d}}{{\rm d}t} \big\|(\vr^N + Z^N)|\vu^N|^2\big\|_{L^1(\Omega)} + \mu \|\Grad \vu^N\|_{L^2(\Omega)}^2 
\end{equation}
$$
+ (\mu+\lambda)\|\Div \vu^N\|_{L^2(\Omega)}^2 = \int_\Omega \Pi_\delta(\vr^N,Z^N) \Div \vu^N \dx.
$$
We shall now express the integral containing $ \Pi_\delta(\vr^N,Z^N) \Div \vu^N$. To this end we introduce the Helmholtz energy function 
${\mathcal H}_\delta$ as a convenient solution of the first order differential equation (\ref{HPODE}), namely 
\begin{equation}\label{eq5.4}
{\mathcal H}_\delta(\vr,Z)= H_{P_\delta}(\vr,Z)  + h_\delta (\vr, Z),
\end{equation}
where  $h_\delta(\vr,Z)= \frac \delta{B-1}(\vr^B+Z^B +\frac 12 \vr^2 Z^{B-2} + \frac 12 Z^2 \vr^{B-2})$,
and $H_{P_\delta}$ is defined in (\ref{HP}).



Finally, we multiply equation (\ref{eq3.2})$_{\vr^N,\vu_N}$  by $\partial_\vr{\mathcal H}_\delta (\vr^N,Z^N)$ and equation (\ref{eq3.3})$_{Z^N\vu^N}$ by $\partial_Z{\mathcal H}_\delta (\vr^N,Z^N)$,
add together and integrate over $\Omega$ in order to deduce
\begin{equation} \label{eq5.5}
\int_\Omega \Pi_{\delta}(\vr^N, Z^N)\Div \vu^N \dx = -  \frac{{\rm d}}{{\rm d}t}\int_\Omega {\mathcal H}_\delta (\vr^N,Z^N) \dx 
\end{equation}
$$
 -\varepsilon \int_\Omega \Big(\frac{\partial^2 {\mathcal H_\delta}}{\partial \vr^2}(\vr^N,Z^N)|\Grad \vr^N|^2
 + 2\frac{\partial^2 {\mathcal H_\delta}}{\partial \vr \partial Z}(\vr^N,Z^N)\Grad \vr^N \cdot \Grad Z^N
$$
$$
+ \frac{\partial^2 {\mathcal H_\delta}}{\partial Z^2}(\vr^N,Z^N)|\Grad Z^N|^2  \Big)\dx.
$$
Putting together \eqref{eq5.3}, \eqref{eq5.5}  and employing the decomposition (\ref{eq5.4}) we get
\begin{equation} \label{eq5.6}
\begin{aligned} 
 &\frac{{\rm d}}{{\rm d}t} \Big(\frac 12 \Big(\|(\vr^N + Z^N)|\vu^N|^2\|_{L^1(\Omega)} + \|\vr^N,Z^N\|_{L^2(\Omega)}^2 \Big) \\
+& \int_\Omega {\mathcal H}_\delta (\vr^N,Z^N) \dx\Big)\\  
+& \varepsilon \int_\Omega 1_{\{\vr^N+Z^N\ge K\}}(\vr^N,Z^N)\Big(\frac{\partial^2 {H_{P_\delta}}}{\partial \vr^2}(\vr^N,Z^N)|\Grad \vr^N|^2 \\ 
+& 2\frac{\partial^2 {H_{P_\delta}}}{\partial 
\vr \partial Z}(\vr^N,Z^N)\Grad \vr^N \cdot \Grad Z^N+ \frac{\partial^2 {H_{P_\delta}
}}{\partial Z^2}(\vr^N,Z^N)|\Grad Z^N|^2  \Big)\dx \\
+& \varepsilon \int_\Omega 1_{\{\vr^N+Z^N< K\}}(\vr^N,Z^N)\Big(\frac{\partial^2 {H_{P_\delta}}}{\partial \vr^2}(\vr^N,Z^N)|\Grad \vr^N|^2 \\ 
+& 2\frac{\partial^2 {H_{P_\delta}}}{\partial 
\vr \partial Z}(\vr^N,Z^N)\Grad \vr^N \cdot \Grad Z^N+ \frac{\partial^2 {H_{P_\delta}
}}{\partial Z^2}(\vr^N,Z^N)|\Grad Z^N|^2  \Big)\dx \\
+&  \varepsilon\delta B \int_\Omega \Big(\frac{\partial^2 {h_{\delta}}}{\partial \vr^2}(\vr^N,Z^N)|\Grad \vr^N|^2 + 2\frac{\partial^2 {h_{\delta}}}{\partial 
\vr \partial Z}(\vr^N,Z^N)\Grad \vr^N \cdot \Grad Z^N \\
+& \frac{\partial^2 {h_{\delta}
}}{\partial Z^2}(\vr^N,Z^N)|\Grad Z^N|^2  \Big) \dx \\
+&   \int_\Omega \Big(\mu |\Grad \vu^N|^2 + (\mu+\lambda)|\Div \vu^N|^2\Big) \dx = 0.
\end{aligned}
\end{equation}

Recalling observation after (\ref{eq3.1}) and assumption (\ref{eq2.3a-}) (or rather its consequence (\ref{eq2.3a})) together
with the uniform bound $0\le \underline{a}\vr^N \leq Z^N\le \overline a\vr^N$, we see that we may chose
$K=K(\delta)>1$ so large that the integral multiplied by $\ep$ containing the characteristic function $ 1_{\{\vr^N+Z^N\ge K\}}$ (the third and the fourth line) will be "absorbed" by the seventh and the eighth line.
Now, we choose number $\Sigma(\delta)$ so large that the integral multiplied by $\ep$ containing the characteristic function $ 1_{\{\vr^N+Z^N< K\}}$ (the fifth and the sixth line)
will be absorbed in the integral  $\ep \Sigma\int_\Omega|\nabla \vr^N,\nabla Z^N|^2{\rm d}x$. This is possible due to estimate (\ref{eq2.3a}). Therefore, summing up
\eqref{eq5.6} and $\Sigma$-multiple of inequality \eqref{eq5.2}, we end up, after an application of the Gronwall lemma,
\begin{equation} \label{eq5.7}
\|\vr^N,Z^N\|_{L^\infty(I;L^B(\Omega))} 
+ \|(\vr^N + Z^N)|\vu^N|^2\|_{L^\infty(I;L^1(\Omega))} 
\end{equation}
$$
+ \|\vu^N\|_{L^2(I;W^{1,2}(\Omega))} 
+ \varepsilon \Big(\|\Grad \vr^N,\Grad Z^N\|^2_{L^2(I;L^2(\Omega))} 
$$
$$ 
+ \int_0^T\int_\Omega \Big(\big|\nabla(\vr^N)\big|^2+\big|\nabla (Z^N)\big|^2\Big)\Big((\vr^N)^{B-2} + (Z^N)^{B-2}\Big)\dx \dt \Big) \leq C,
$$
where the constant $C$ is independent of $N$ and $\varepsilon$ (but blows up when $\delta \to 0^+$).

With estimate (\ref{eq5.7}) at hand, we may shift in (\ref{eq3.2}) and (\ref{eq3.3}) the nonlinear terms to the right-hand side. We consider them as homogenous Neumann problems 
with right-hand sides $-{\rm div}(\vr^N\vu^N)$ and $-{\rm div}(Z^N\vu^N)$, respectively, and get estimates independent of $N$ for
 \begin{equation}\label{eq5.7+}
\partial_t\vr^N, \partial_t Z^N, \Grad^2\vr^N, \Grad^2 Z^N\;\mbox{in $L^r(I\times\Omega)$ with some $r>1$ },
\end{equation}
$$
\Grad\vr^N, \Grad Z^N \;\mbox{in $L^q(I; L^2(\Omega)$ with some $q>2$ }
$$
via the maximal parabolic regularity, cf. \cite{FNP}.

The estimates evoked above provide enough standard  compactness in order to pass to the limit
$N\to\infty$ in equations (\ref{eq3.2}--\ref{eq3.4}) in the same way as in the mono-fluid compressible Navier--Stokes equations. 
Indeed, we have weak (eventually weak-*) convergence of sequences $(\vr^N, Z^N,$ $\partial_t\vr^N, \partial_t Z^N,$  $\nabla \vr^N, \nabla Z^N,$  $\nabla^2 \vr^N,\nabla^2 Z^N,$ $\vu^N,$ $\nabla\vu^N)$ and $\nabla (\vr^N)^{B/2}, \nabla (Z^N)^{B/2})$, to corresponding weak limits  in the functional spaces, where they are bounded, cf. (\ref{eq5.7}--\ref{eq5.7+}).
Almost everywhere convergence of $(\vr^N,Z^N, \nabla\vr^N,\nabla Z^N)$ 
as well as $C_{\rm weak}([0,T]; L^B(\Omega))$ convergence  of $(\vr^N,Z^N)$ to  $(\vr,Z)$, and $L^\infty(I;L^{\frac{2B}{B+1}}(\Omega))$ convergence  $(\vr^N\vu^N, Z^N\vu^N)$ to $(\vr\vu, Z\vu)$ can be obtained by using classical compactness arguments (Rellich--Kondrachov, Lions--Aubin--Simon and Arzel\`a--Ascoli type lemmas). This is enough to pass to the limit $N\to\infty$ in equations  (\ref{eq3.2}--\ref{eq3.3})$_{\vr^N,Z^N,\vu^N}$ in order to get the regularized continuity equation  (\ref{eq3.2}) and (\ref{eq3.3}) for the limiting functions $(\vr,\vu)$ and (Z,\vu), respectively.

Returning to the momentum equation (\ref{eq3.4})$_{\vr^N,Z^N,\vu^N}$ in order to justify equi-continuity we deduce using an Arzel\`a--Ascoli type argument the $C_{\rm weak}(\overline I; L^{q}(\Omega;R^3))$ convergence of
 $(\vr^N+Z^N)\vu^N$ to $(\vr+Z)\vu$ with a convenient $q>1$ and consequently
together with weak convergence of $\vu^N$ in $L^2(I,W^{1,2}(\Omega;R^3))$ finally
at least the $L^1(Q_T;R^9)$ weak convergence of the convective term to $\vr\vu\otimes\vu$.

Using in the limit equations (\ref{eq3.2}--\ref{eq3.3})$_{\vr,Z,\vu}$ test-functions $\vr$ and $Z$, respectively, we get identity (\ref{eq5.2}) for weak limits $(\vr,Z,\vu)$ (i.e. (\ref{eq5.2}) with
$(\vr, Z, \vu)$ on place of $(\vr^N, Z^N, \vu^N)$). Comparison of (\ref{eq5.2})$_{\vr^N, Z^N, \vu^N}$ with (\ref{eq5.2})$_{\vr, Z, \vu}$ yields strong convergence of $(\nabla \vr^N,\nabla Z^N)$ towards $ (\nabla \vr,\nabla Z)$ in $(L^2(Q_T;R^3))^2$. This is the last element we need to pass to the limit in the momentum equation (\ref{eq3.4}) in order to get
\begin{equation} \label{eq5.10}
\begin{aligned}
&\int_0^T \int_\Omega \big((\vr+Z)\vu \cdot \partial_t \vcg{\varphi} + (\vr+Z) (\vu\otimes \vu): \Grad \vcg{\varphi} + P(\vr,Z) \Div \vcg{\varphi}\big) \dx \dt \\
& - \varepsilon \int_0^T \int_\Omega \big(\Grad (\vr+Z)\cdot \Grad \vu\big)\cdot \vcg{\varphi}) \dx\dt \\
= &\int_0^T\int_\Omega (\mu \Grad \vu :\Grad \vcg{\varphi} + (\mu+\lambda) \Div \vu \, \Div \vcg{\varphi} \big) \dx \dt - \int_\Omega \vc{m}_0 \cdot \vcg{\varphi}(0,\cdot) \dx 
\end{aligned} 
\end{equation} 
for any function $\vcg{\varphi} \in C_c^1([0,T)\times \Omega;R^3)$.

Last but not least, due to the $C_{\rm weak}(\overline I;L^\gamma(\Omega))$ convergence of $(\vr^N,Z^N)$ combined with the theorem on Lebesgue points, we deduce from (\ref{eq5.1}) for all $t\in \overline I$ and a.a. $x \in \Omega$
\begin{equation}\label{eq5.10+}
\vr(t,x)\ge 0,\; \underline a \vr(t,x)\le Z(t,x)\le\overline a \vr(t,x).
\end{equation}

Finally, we can also pass to the limit in the integrated form of the energy
inequality (\ref{eq5.6}) --- note that we cannot prove the strong convergence of several terms and we must use the lower weak semi-continuity of those terms induced by the convex
functionals --- this will be notably the case of functionals $\int_0^T\int_\Omega (\mu |\Grad \vu^N|^2$ $+ (\mu+\lambda)|\Div \vu^N|^2) \dx\,{\rm d}t$, $\int_\Omega((\vr^N)^B+(Z^N)^B + \frac 12(\vr^N)^2 (Z^N)^{B-2} + \frac 12 (Z^N)^2 (\vr^N)^{B-2})\,{\rm d}x$, and also
$\int_0^T\int_\Omega (|\nabla(\vr^N)^{B/2}|^2$ \linebreak $+|\nabla (Z^N)^{B/2}|^2) \dx \dt$  ---
 and we get therefore only an inequality:  
\begin{equation} \label{eq5.11}
\Big(\frac 12\|(\vr + Z)|\vu|^2(t)\|_{L^1(\Omega)} 
+ \int_\Omega {\mathcal H}_\delta (\vr,Z)(t) \dx\Big) 
\end{equation}
$$
+  \int_0^t\int_\Omega \Big(\mu |\Grad \vu|^2 + (\mu+\lambda)|\Div \vu|^2\Big) \dx \,{\rm d}\tau \\
+ \varepsilon \int_0^t\int_\Omega \Big(\frac{\partial^2 {\mathcal H_\delta}}{\partial \vr^2}(\vr,Z)|\Grad \vr|^2 
$$
$$
 + 2\frac{\partial^2 {\mathcal H_\delta}}{\partial \vr \partial Z}(\vr,Z)\Grad \vr \cdot \Grad Z
+ \frac{\partial^2 {\mathcal H_\delta}}{\partial Z^2}(\vr,Z)|\Grad Z|^2  \Big)\dx \,{\rm d}\tau
$$
$$
 \leq \frac 12 \int_\Omega \frac{|\vc{m}_0|^2}{Z_0 + \vr_0} \dx + \int_\Omega \mathcal{H}_\delta (\vr_0,Z_0)\dx
$$
for a.a. $t\in (0,T)$.

\subsection{Limit passage $\varepsilon \to 0^+$} \label{se6}

From the previous section we have at our disposal a solution $(\vr_\ep, Z_\ep, \vu_\ep)$ of system (\ref{eq3.2}--\ref{eq3.3}), (\ref{eq5.10}), obeying energy inequality (\ref{eq5.11}) and belonging to class (\ref{eq5.7+}):
\begin{equation}\label{clep}
\begin{aligned}
(\vr,Z)&\in \big(L^\infty(I; L^B(\Omega))\cap L^r(I,W^{2,r}(\Omega))\cap L^2(I;W^{1,2}(\Omega))\big)^2,\\
 (\partial_t\vr,\partial_t Z)&\in \big(L^r(I\times \Omega)\big)^2, \;
\vu\in L^2(I,W^{1,2}(\Omega;R^3)),
\end{aligned}
\end{equation}
and for all $t\in \overline I$ and a.a. $x \in \Omega$
$$
\underline a \vr\le Z(t,x)\le\overline a\vr(t,x).
$$
Moreover, inequalities (\ref{eq5.7}) and (\ref{eq5.11}) provide, in particular, the following estimates uniform with $\ep$:
	\begin{multline}\label{estep1}
	\|(\vr_\ep, Z_\ep)\|_{L^\infty(I;L^B(\Omega))}+\sqrt\ep\|\nabla\vr_\ep,\nabla Z_\ep\|_{L^2(Q_T)}\\
	+ \|\vu_\ep\|_{L^2(I;W^{1,2}(\Omega))}\le C,
	\end{multline}
	\begin{equation}\label{estep2}
	\|\vr_\ep|\vu_\ep|^2\|_{L^\infty(I;L^1(\Omega))}\le C.
	\end{equation}
This is enough to multiply equation (\ref{eq3.2})$_{\vr_\ep,\vu_\ep}$ by $\partial_\vr b(\vr_\ep,Z_\ep)$ and equation (\ref{eq3.3})$_{Z_\ep,\vu_\ep}$  by $\partial_Z b(\vr_\ep,Z_\ep)$, where $b(\vr,Z)= b_\delta(\vr,Z)= \frac {Z^2}{\vr+\delta}$, $\delta>0$, to arrive at the identity
\begin{equation}\label{Z2r-}
\partial_t b_\delta(\vr_\ep,Z_\ep)+{\rm div}\Big(b_\delta(\vr_\ep,Z_\ep)\vu_\ep\Big)+
\delta\frac{Z_\ep^2}{(\vr_\ep+\delta)^2}{\rm div}\,\vu_\ep
\end{equation}
$$
+ \frac{\partial^2 { b_\delta}}{\partial \vr^2}(\vr_\ep,Z_\ep)|\Grad \vr_\ep|^2  + 2\frac{\partial^2 { b_\delta}}{\partial \vr \partial Z}(\vr_\ep,Z_\ep)\Grad \vr_\ep \cdot \Grad Z_\ep+ \frac{\partial^2 {b_\delta}}{\partial Z^2}(\vr_\ep,Z_\ep)|\Grad Z_\ep|^2 =0
$$
a.a. in $Q_T$. Integrating this identity over $(0,\tau)\times\Omega$ while employing the fact that $b_\delta$ is a convex function, we obtain for all $\tau\in \overline I$
$$
\int_\Omega\frac {Z_\ep^2}{\vr_\ep+\delta}(\tau,x)\,{\rm d}x+\delta\int_0^\tau\int_{\Omega} \frac{Z_\ep^2}{(\vr_\ep+\delta)^2}{\rm div}\,\vu_\ep\,{\rm d}x\le\int_\Omega\frac  {Z_0^2}{\vr_0+\delta}\,{\rm d}x.
$$
Now, we can use the Lebesgue dominated convergence theorem in the limit $\delta\to 0$
(see the uniform estimate in the second line of (\ref{clep})) to get
\begin{equation}\label{Z2r}
\mbox{for all $\tau\in \overline I$},\;
\int_\Omega Z_\ep s_\ep(\tau,x)\,{\rm d}x\le\int_\Omega  \vr_0
s_0^2\,{\rm d}x,
\end{equation}
where $s_\ep=Z_\ep/\vr_\ep$, $s_0=Z_0/\vr_0$  is defined as in (\ref{sn}), according to convention (\ref{conv}).

	We are now in position to announce a counterpart of Proposition \ref{L2} for
	regularized versions (\ref{eq3.2}--\ref{eq3.3})$_{\vr_\ep, Z_\ep,\vu_\ep}$
	of continuity equations. Its proof is essentially the same  as the proof 
	of Proposition \ref{L2} and will be therefore skipped.

\begin{prop}\label{p4}
Let $\vr_\ep$, $Z_\ep$ and $\vu_\ep$ from the class  (\ref{clep}) solve equations (\ref{eq3.2}--\ref{eq3.3})$_{\vr_\ep, Z_\ep,\vu_\ep}$. Suppose further that they satisfy bounds
(\ref{Z2r}) and (\ref{estep1}).
Then we have:
\begin{description}
\item{\it 1.} Up to a subsequence (note relabeled),
\begin{equation} \label{eq3.7}
\begin{aligned}
(\vr_\ep, Z_\ep) &\rightharpoonup_* (\vr, Z)\;\mbox{in $L^\infty(I;L^B(\Omega;R^2))$}, \\
\vu_\ep &\rightharpoonup \vu \; \text{ weakly in } L^2(I; W^{1,2}(\Omega;R^3)),
\end{aligned}
\end{equation}
where $(\vr,Z)\in C(\overline I,L^1(\Omega))$ and each couple $(\vr,\vu)$ and $(Z,\vu)$ solve continuity equation (\ref{ce1}).
\item{\it 2.} There holds
\begin{equation} \label{eq3.8}
\mbox{for all $\tau\in\overline I$}\; \int_\Omega \vr_\ep |s_\ep-s|^p(\tau,\cdot) \dx \to 0,\;\mbox{with any $1\le p<\infty$},
\end{equation}
where $s_\ep(t,x)=Z_\ep(t,x)/\vr_\ep(t,x)$ and $s(t,x)=Z/\vr(t,x)$ are defined for any $t\in \overline I$ in agreement with (\ref{conv}).
They are bounded for any $\tau\in \overline I$ by $\underline a$ from below and by $\overline a$
from above for a.a. $x\in \Omega$.
\end{description}
\end{prop}

We now aim at proving better than $L^1$-bound for the pressure. First, we recall the properties of the Bogovskii operator, see e.g. Galdi \cite{Ga}, or \cite[Appendix]{FeNoB}, or \cite[Chapter 3]{NoSt}. This operator is  a bounded map the space $L^p_0(\Omega) = \{f \in L^p(\Omega)| \int_\Omega f \dx = 0\}$ to $W^{1,p}(\Omega;R^3)$ such that ($\Omega$ must have at least Lipschitz boundary)
\begin{equation}\label{eq6.1}
\Div {\mathcal B}(f) = f, \; \|{\mathcal B}(f)\|_{W^{1,p}(\Omega)} \leq C \|f\|_{L^p(\Omega)},\;
\|{\mathcal B}({\rm div}\,{\vc g})\|_{L^{q}(\Omega)} \leq C \|{\vc g}\|_{L^q(\Omega)}
\end{equation}
for any $1<p<\infty$, $1<q<\infty$, where $\vc g\in L^q(\Omega)$, ${\rm div}\,\vc g\in L^q(\Omega)$ and $\vc g\cdot\vc n|_{\partial\Omega}=0$ in the sense of normal traces. Therefore, using as test function in \eqref{eq5.10} $\psi(t) {\mathcal B}\big(\vr-\frac{1}{|\Omega|}\int_\Omega \vr \dx\big)$, where $\psi \in C_c^1([0,T))$,
$\psi(0)=0$ belongs to a convenient family of non-negative functions,
we end up after standard computations as for the compressible Navier--Stokes equations (here we use \eqref{eq2.2} from Hypothesis (H3))
\begin{equation}  \label{eq6.2}
\int_0^T \int_\Omega \big(\vr^{\gamma+1} + \delta (\vr^{B+1}+ \vr Z^B)\big)\dx \leq C(\delta,\vc{m}_0,\vr_0, Z_0).
\end{equation}
Hence we find 
\begin{equation}\label{eq6.2+}
\|\Pi_\delta(\vr_\varepsilon,Z_\varepsilon)\|_{L^{\frac{B+1}{B}} (I\times \Omega)} \leq C
\end{equation}
with $C$ independent of $\varepsilon$.


 Now, we may let $\varepsilon \to 0^+$ in equations (\ref{eq3.2}--\ref{eq3.3})$_{\vr_\ep, Z_\ep,\vu_\ep}$, (\ref{eq5.1})$_{\vr_\ep, Z_\ep,\vu_\ep}$. We notice that the limit passage in the convective terms can be performed as in the case of the mono-fluid compressible Navier--Stokes equations. Indeed, seeing that 
\begin{equation}\label{conv1}
(\vr_\ep,Z_\ep)\to (\vr,Z) \;\mbox{in $(C_{\rm weak}(\overline I;L^{B}(\Omega)))^2$ }
\end{equation}
(as one can show by means of the  Arzel\`a--Ascoli type argument from equation (\ref{eq3.2}--\ref{eq3.3})$_{\vr_\ep,Z_\ep,\vu_\ep}$ and uniform bounds (\ref{estep1}--\ref{estep2})), we deduce from the compact embedding 
$L^B(\Omega)\hookrightarrow\hookrightarrow W^{-1,2}(\Omega)$ and from $\vu_\ep\rightharpoonup\vu$ in $L^2(I;W^{1,2}(\Omega))$ the weak-* convergence
$(\vr_\ep\vu_\ep,Z_\ep\vu_\ep)\rightharpoonup_* (\vr\vu,Z\vu)$ in $L^\infty(I;L^{\frac{2B}{B+1}}(\Omega))$ that may be consequently improved thanks to momentum equation (\ref{eq5.10})$_{\vr_\ep,Z_\ep,\vu_\ep}$
and estimates (\ref{estep1}--\ref{estep2}), (\ref{eq6.2+}) to 
\begin{equation}\label{conv2}
(\vr_\ep+Z_\ep)\vu_\ep\to(\vr+Z)\vu\;\mbox{ in $C_{\rm weak}(\overline I; L^{\frac{2B}{B+1}}(\Omega;R^3))$}
\end{equation}
 again by the Arzel\`a--Ascoli type argument. 
With this observation at hand,
employing compact embedding  $ L^{\frac{2B}{B+1}}(\Omega)\hookrightarrow\hookrightarrow W^{-1,2}(\Omega)$ and $\vu_\ep\rightharpoonup\vu$ in $L^2(I;W^{1,2}(\Omega;R^3))$ we infer   that
$$
(\vr_\ep+Z_\ep)\vu_\ep\to (\vr+Z)\vu \quad \text{ in } \quad L^2(0,T,W^{-1,2}(\Omega;R^3))
$$ 
and consequently 
\begin{equation}\label{conv3}
(\vr_\ep+Z_\ep)\vu_\ep\otimes\vu_\ep
 \rightharpoonup (\vr+Z)\vu\otimes\vu\;\mbox{weakly e.g. in $L^1(Q_T;R^{9})$,}
\end{equation}
at least for a chosen subsequence (not relabeled). 
This reasoning remains valid under limitation $B>3/2$ imposed by the use of compact embeddings above. Resuming (and realizing that all terms multiplied by $\ep$ will disappear in the limit again by virtue of (\ref{estep1}--\ref{estep2})) we get
\begin{equation} \label{eq6.3}
\begin{aligned}
\int_0^T \int_\Omega \big(\vr \partial_t \psi + \vr \vu \cdot \Grad \psi\big) \dx \dt  + \int_\Omega \vr_0 \psi(0,\cdot) \dx &=0, \\
\int_0^T \int_\Omega \big(Z \partial_t \psi + Z \vu \cdot \Grad \psi\big) \dx \dt + \int_\Omega Z_0 \psi(0,\cdot) \dx &=0
\end{aligned}
\end{equation}
for any $\psi \in C^1_c([0,T) \times \Ov{\Omega})$, 
\begin{equation} \label{eq6.4}
\begin{aligned}
&\int_0^T \int_\Omega \big((\vr+Z)\vu \cdot \partial_t \vcg{\varphi} + (\vr+Z) (\vu\otimes \vu): \Grad \vcg{\varphi} + \overline{\Pi_\delta(\vr,Z)} \Div \vcg{\varphi}\big) \dx \dt \\
= &\int_0^T\int_\Omega (\mu \Grad \vu :\Grad \vcg{\varphi} + (\mu+\lambda) \Div \vu \, \Div \vcg{\varphi} \big) \dx \dt - \int_\Omega \vc{m}_0 \cdot \vcg{\varphi}(0,\cdot) \dx
\end{aligned} 
\end{equation}	
for any $\vcg{\varphi} \in C^\infty_c([0,T)\times \Omega;R^3)$. Above, the bar over $\Pi_\delta(\vr,Z)$ denotes the weak limit of $\Pi_\delta(\vr_\varepsilon, Z_\varepsilon)$ in at least $L^1(Q_T)$. We will use this notation in general for weak limits of sequences (of nonlinear) functions of quantities $(\vr_\ep, Z_\ep)$ or even of  $(\vr_\ep, Z_\ep,\vu_\ep)$ throughout the rest of the paper.

We may also pass to the limit in the energy inequality to get  
\begin{equation} \label{eq6.5}
\begin{aligned} 
&\frac 12 \|(\vr + Z)|\vu|^2(t)\|_{L^1(\Omega)} + \int_\Omega {\mathcal H}_\delta (\vr,Z)(t) \dx
\\ 
&+  \int_0^t\int_\Omega \Big(\mu |\Grad \vu|^2 + (\mu+\lambda)|\Div \vu|^2\Big) \dx \,{\rm d}\tau \\
& \leq \frac 12 \int_\Omega \frac{|\vc{m}_0|^2}{Z_0 + \vr_0} \dx + \int_\Omega \mathcal{H}_\delta (\vr_0,Z_0)\dx.
\end{aligned}
\end{equation}

Now, we may write
$$
\Pi_\delta(\vr_\varepsilon,Z_\varepsilon) = \Pi_\delta(\vr_\varepsilon, \vr_\varepsilon s_\varepsilon) 
= \Pi_\delta (\vr_\varepsilon,\vr_\varepsilon s_\varepsilon) - \Pi_\delta(\vr_\varepsilon, \vr_\varepsilon s) + \Pi_\delta(\vr_\varepsilon,  \vr_\varepsilon s).
$$
We have, due to Proposition \ref{p4} and Hypothesis (\ref{eq2.5-}) (or rather its consequence (\ref{eq2.5})),
\begin{multline*}
\lim_{\varepsilon \to 0} \Big|\int_0^T \int_\Omega \big(\Pi_\delta(\vr_\varepsilon,\vr_\varepsilon s_\varepsilon) - \Pi_\delta(\vr_\varepsilon,\vr_\varepsilon s)\big)\dx\dt\Big|  \\
\leq c(\delta) \lim_{\varepsilon \to 0} \Big(\int_0^T \int_\Omega L_{P}(\vr_\ep)|s_\varepsilon -s| \dx\dt  + \int_0^T \int_\Omega \vr_\epsilon^B |s_\varepsilon^B -s^B| \dx\dt\Big)  = 0.
\end{multline*}
Hence,
\begin{equation} \label{eq6.6}
\overline{\Pi_\delta(\vr,Z) }:=
{w-\lim_{\varepsilon \to 0}} \Pi_\delta(\vr_\varepsilon,Z_\varepsilon) = {w-\lim_{\varepsilon \to 0}} \Pi_\delta (\vr_\varepsilon,\vr_\varepsilon s) =: \Ov{\Ov{\Pi_\delta (\vr,\vr s)}},
\end{equation} 
{where notation ${w-\lim_{\varepsilon \to 0}}$ means weak limit in $L^1(Q_T)$.}
Therefore we may rewrite the momentum equation in the form
\begin{equation} \label{eq6.7}
\begin{aligned}
&\int_0^T \int_\Omega \big((\vr+Z)\vu \cdot \partial_t \vcg{\varphi} + (\vr+Z) (\vu\otimes \vu): \Grad \vcg{\varphi} + \Ov{\Ov{\Pi_\delta (\vr,\vr s)}} \, \Div \vcg{\varphi}\big) \dx \dt \\
= &\int_0^T\int_\Omega (\mu \Grad \vu :\Grad \vcg{\varphi} + (\mu+\lambda) \Div \vu \, \Div \vcg{\varphi} \big) \dx \dt - \int_\Omega \vc{m}_0 \cdot \vcg{\varphi}(0,\cdot) \dx.
\end{aligned} 
\end{equation}
Since $s$ is now fixed, we may apply to the problem the theory available for the
 generally non-monotone pressure of one variable from Feireisl \cite{Fe2002}. 

To this aim, we first recall the {\it effective viscous flux identity} which in our situation has the form
\begin{prop} \label{p5}
We have, possibly for a subsequence $\varepsilon \to 0^+$, the following identity
\begin{equation} \label{eq6.8}
\Ov{\Ov{\Pi_\delta (\vr,\vr s) \vr}} -(2\mu+\lambda) \Ov{\vr \Div \vu} = \Ov{\Ov{\Pi_\delta (\vr,\vr s)}} \vr -(2\mu+\lambda) \vr \, \Div \vu
\end{equation}
fulfilled a.a. in $I\times \Omega$, where $\Ov{\Ov{\Pi_\delta (\vr,\vr s) \vr}} = {w-}\lim_{\varepsilon \to 0} \Pi_\delta(\vr_\varepsilon, \vr_\varepsilon s)\vr_\varepsilon$.
\end{prop}

\begin{proof}
We denote by $\nabla\Delta^{-1}$ the pseudodifferential operator with Fourier symbol $\frac {{\rm i}\xi}{|\xi|^2}$ and by ${\tn R}$ the Riesz transform
with Fourier symbol $\frac {\xi\otimes\xi}{|\xi|^2}$.
Following Lions \cite{L4}, we shall use in the approximating momentum equation (\ref{eq5.10})$_{\vr_\ep,Z_\ep,\vu_\ep}$ test function
\begin{equation}\label{test1}
\varphi(t,x)=\psi(t)(\nabla\Delta^{-1}(\vr_\ep\phi))(t,x),\;\;\psi\in C^1_c(0,T),\;\phi\in C^1_c(\Omega)
\end{equation}
and in the limiting momentum equation (\ref{eq6.4}) test function 
\begin{equation}\label{test2}
\varphi(t,x)=\psi(t)(\nabla\Delta^{-1}(\vr\phi))(t,x),\;\;\psi\in C^1_c(0,T),\;\phi\in C^1_c(\Omega),
\end{equation}
subtract both identities and perform the limit passage $\ep\to 0$. This is a laborious, but nowadays standard calculation (whose details, for "simple" compressible Navier--Stokes equations, can be found e.g. in 
\cite[Lemma 3.2]{FNP}, \cite[Chapter 3]{NoSt}, \cite{EF70} or \cite[Chapter 3]{FeNoB}) leading
 to the identity
\begin{equation}\label{ddd!}
\int_0^T\intO{\psi\phi\Big(\overline{\overline{\Pi_\delta(\vr,\vr s)}}
-(2\mu +\lambda){\rm div}\,\vu\Big)\vr}\,{\rm d}t
\end{equation}
$$
-\int_0^T\intO{\psi\phi\Big(\overline{\overline{\Pi_\delta(\vr,\vr s)\vr}}
-(2\mu +\lambda)\overline{\vr\,{\rm div}\,\vu}\Big)}\,{\rm d}t
$$
$$
=
\int_0^T\intO{\psi\vu\cdot\Big(\vr {\tn R}\cdot((\vr+Z)\vu\phi)-(\vr+Z)\vu\cdot{\tn R}(\vr\phi)\Big)
}\,{\rm d}t
$$
$$
-
\lim_{\ep\to 0}\int_0^T\intO{\psi\vu_\ep\cdot\Big(\vr_\ep {\tn R}\cdot(\vr_\ep+Z_\ep)\vu_\ep\phi)-(\vr_\ep+ Z_\ep)\vu_\ep\cdot{\tn R}(\vr_\ep\phi)\Big)
}\,{\rm d}t.
$$
This process involves several integrations by parts and exploits continuity equations in form (\ref{eq6.3}) and regularized continuity equations in form
 (\ref{eq3.2}--\ref{eq3.3})$_{\vr_\ep,Z_\ep,\vu_\ep}$ in the same way as in the
mono-fluid theory. As in the mono-fluid theory,
{the essential observation for getting (\ref{ddd!}) is the fact that the map $\vr\mapsto\varphi$ defined above is a linear and continuous from
$L^p(\Omega)$ to $W^{1,p}(\Omega)$, $1<p<\infty$ as a consequence of classical   H\"ormander--Michlin's multiplier theorem of harmonic analysis.}
The most non trivial moment in this process is to show that the right-hand side of identity (\ref{ddd!}) is $0$. To see it, we repeat the reasoning \cite{FNP}
adapted to this situation.
We first realize  that the $(C_{\rm weak}(I, L^B(\Omega)))^2$-convergence of $(\vr_\ep, Z_\ep)$ and 
$C_{\rm weak}([0,T], L^{\frac{2B}{B+1}}(\Omega;R^3))$-convergence of $(\vr_\ep+Z_\ep)\vu_\ep$ evoked in (\ref{conv1}--\ref{conv2})
imply, in particular,
\begin{equation}\label{cvep!}
\mbox{for all $t\in [0,T]$},\;(\vr_\ep, Z_\ep)(t)\rightharpoonup(\vr, Z)(t)\;\mbox{in e.g. $(L^B(\Omega))^2$},
\end{equation}
$$
(\vr_\ep+Z_\ep)\vu_\ep(t) \rightharpoonup(\vr+Z)\vu(t)\;\mbox{in 
$L^{\frac {2B}{B+1}}(\Omega;R^3)$}.
$$
Since ${\tn R}$ is a continuous operator from $L^p(R^3)$ to $L^p(R^3)$, $1<p<\infty$, we  have the same type of convergence for sequences ${\tn R}[\vr_\ep(t)]$,
${\tn R}[Z_\ep(t)]$ and ${\tn R}[(\vr_\ep+Z_\ep)\vu_\ep(t)]$ to their respective limits
 ${\tn R}[\vr(t)]$,
${\tn R}[Z(t)]$ and ${\tn R}[(\vr+Z)\vu(t)]$.

 At this stage we  apply  to the above situation Proposition \ref{rieszcom}, and get
 $$
[\vr_\ep {\tn R}\cdot(\vr_\ep+Z_\ep)\vu_\ep\phi)-\vr_\ep\vu_\ep\cdot{\tn R}(\vr_\ep\phi)](t)\rightharpoonup
[\vr {\tn R}\cdot(\vr\vu\phi)-\vr\vu\cdot{\tn R}(\vr\phi)](t)
$$
for all $t\in [0,T]$ (weakly) in $L^{\frac{2B}{B+3}}(\Omega)$.
In view of compact embedding $L^{\frac{2B}{B+1}}(\Omega)\hookrightarrow\hookrightarrow W^{-1,2}(\Omega)$, 
and the boundedness of $\|\vr_\ep {\tn R}\cdot(\vr_\ep+Z_\ep)\vu_\ep)-(\vr_\ep+Z_\ep)\vu_\ep\cdot{\tn R}(\vr_\ep)\|_{W^{-1,2}(\Omega)}$  in $L^\infty(I)$,
we infer, in particular,
$$
\vr_\ep {\tn R}\cdot((\vr_\ep+Z_\ep)\vu_\ep\phi)-(\vr_\ep+Z_\ep)\vu_\ep\cdot{\tn R}(\vr_\ep\phi)\to\vr {\tn R}\cdot((\vr+Z)\vu\phi)-(\vr+Z)\vu\cdot{\tn R}(\vr\phi)
$$
in $L^2(0,T;W^{-1,2}(\Omega))$.
Recalling the $L^2(I;W^{1,2}(\Omega))$-weak convergence of $\vu_\ep$ we get the desired result (\ref{ddd!}). This completes the proof of Proposition \ref{p5}.

We realize for the further later reference, that the part of argumentation starting from (\ref{cvep!}) requires $B>9/2$. 
\end{proof}

We are now ready to prove the strong convergence of $\vr_\varepsilon \to \vr$,
more exactly
\begin{equation}\label{a.a.ep}
\vr_\varepsilon \to \vr\;\mbox{a.a. in $Q_T$}.
\end{equation}

We will do it again by mimicking the mono-fluid case, see e.g. \cite[Chapter 7]{NoSt}. To this end, we multiply the regularized continuity equation (\ref{eq3.2})$_{\vr_\ep,\vu_\ep}$ by $b_\delta'(\vr_\ep)$, $\delta>0$, where $b_\delta(\vr)=\vr\ln(\vr+\delta)$ is a (strictly) convex function on $[0,\infty)$. Integrating the resulting identity over $(0,\tau)\times\Omega$, employing convexity of $b_\delta$, reasoning as in (\ref{Z2r-}--\ref{Z2r}), we get after the passage to the limit $\delta\to 0$ and $\ep\to 0$ (in this order)
\begin{equation}\label{*}
\forall \tau\in \overline I,\;\int_{\Omega}\overline{\vr\ln\vr}(\tau,\cdot)\,{\rm d }x
-\int_{\Omega}\vr_0\ln\vr_0\,{\rm d }x
\le\int_0^\tau\int_\Omega\overline{\vr{\rm div}\,\vu}\,{\rm d}x\,{\rm d}t.
\end{equation}
In the above, in agreement with our convention
$\overline{\vr\ln\vr}$ and $\overline{\vr{\rm div}\,\vu}$ are $L^1(Q_T)$ weak limits
 of sequences $\vr_\ep\ln\vr_\ep$ and $\vr_\ep{\rm div}\,\vu_\ep$, respectively.

On the other hand, using the fact that $\vr$ is a renormalized solution of the continuity equation (\ref{eq6.3}) we obtain employing Proposition \ref{p1}  that
\begin{equation}\label{**}
\forall \tau\in \overline I,\;\int_{\Omega}\vr\ln\vr(\tau,\cdot)\,{\rm d }x -\int_{\Omega}\vr_0\ln\vr_0 \,{\rm d }x
=\int_0^\tau\int_\Omega\overline{\vr{\rm div}\,\vu}\,{\rm d}x\,{\rm d}t.
\end{equation}

We may therefore use \eqref{eq6.8}  to conclude that
$$
\int_\Omega (\Ov{\vr \ln \vr}-\vr \ln \vr)(\tau,\cdot) \dx \leq \frac{1}{2\mu+\lambda} \int_0^\tau \int_\Omega (\Ov{\Ov{\Pi_\delta (\vr,\vr s)}} \vr - \Ov{\Ov{\Pi_\delta (\vr,\vr s) \vr}}) \dx \dt.
$$
We now apply our  Hypothesis (H4) (cf. (\ref{eq2.4})) and write 
$$
\Pi_\delta(\vr, \vr s) = \mathcal P^\delta (\vr,s) - \mathcal R^\delta(\vr,s),
$$
where $\vr\mapsto{\mathcal P}^\delta (\vr,s)=(1-\eta_\delta(\vr\sqrt{1+s^2}){\cal P}(\vr,s)$ is non-decreasing for any $s\in[\underline a,\overline a]$, and $\vr\mapsto\mathcal R^\delta (\vr,s):= (1-\eta_\delta(\vr\sqrt{1+s^2}){\cal P}(\vr,s)$ is for any $s\in[\underline a,\overline a]$  non-negative, $C^2([0,\infty))$ uniformly bounded with respect to $s\in[\underline a,\overline a]$ and uniformly  compactly supported in $[0,\infty)$ with respect to $s \in [\underline a,\overline a]$. Recall that the cut-off function $\eta_\delta$ is defined in (\ref{eta}). We employ monotonicity of $\vr\mapsto{\cal P}_\delta(\vr,s)$ through Proposition \ref{p3} in order to infer
\begin{equation} \label{eq6.8a}
\int_\Omega (\Ov{\vr \ln \vr}-\vr \ln \vr)(\tau,\cdot) \dx \leq \frac{1}{2\mu+\lambda} \int_0^\tau \int_\Omega \big(\overline{\Ov{\mathcal R^\delta(\vr,s)\vr}}  - \overline{\Ov{\mathcal R^\delta(\vr,s)}}\vr\big) \dx \dt.
\end{equation}
Here and in what follows $\Ov{\Ov{b(\vr,s)}}$ denotes a weak limit (at least in $L^1(Q_T)$) of sequence $b(\vrd,s)$ (with $s$ a fixed function).

Due to the properties of $\mathcal R^\delta (\cdot,s)$ there exists $\Lambda>0$ such that $z\mapsto \Lambda z \ln z -z\mathcal R^\delta(z,s)$ and $z\mapsto \Lambda z \ln z +\mathcal R^\delta(z,s)$ are convex on $[0,\infty)$ for any $s \in [\underline a,\overline a]$. Consequently,
\begin{multline*}
\int_0^\tau \int_\Omega \big(\overline{\Ov{\mathcal R^\delta(\vr,s)\vr}}  - \overline{\Ov{\mathcal R^\delta(\vr,s)}}\vr\big) \dx \dt \leq  \Lambda \int_0^\tau \int_\Omega (\Ov {\vr \ln \vr}- \vr \ln \vr)\dx \dt \\
 + \int_0^\tau \int_\Omega \big(\mathcal R^\delta(\vr,s)- \Ov {\Ov{ \mathcal R^\delta(\vr,s)}}\big)\vr \dx \dt. 
\end{multline*}
As $\mathcal R^\delta$ is non-negative and $z\mapsto \Lambda z \ln z +\mathcal R^\delta(z,s)$ is convex, we have for any $s \in [\underline a,\overline a]$,
$$
\begin{aligned}
\int_0^\tau \int_\Omega \big(\mathcal R^\delta(\vr,s)&- \Ov {\Ov{ \mathcal R^\delta(\vr,s)}}\big)\vr \dx \dt \\
& \leq \int_{\{(t,x)\in I\times \Omega| \vr(t,x) \leq R_0\}} \big(\mathcal R^\delta(\vr,s)- \Ov {\Ov{ \mathcal R^\delta(\vr,s)}}\big)\vr \dx \dt \\
&\leq \Lambda  \int_{\{(t,x)\in I\times \Omega| \vr(t,x) \leq R_0\}} (\Ov {\vr \ln \vr}- \vr \ln \vr)\vr \dx \dt \\ 
&\leq \Lambda R_0\int_0^t \int_\Omega (\Ov {\vr \ln \vr}- \vr \ln \vr) \dx \dt,
\end{aligned}
$$
where $R_0>1$ is such that $\cup_{s\in[\underline a,\overline a]}{\rm supp}{\cal R}(\cdot,s)\subset [0,R_0]$.
Thus, finally 
\begin{equation}\label{eqeq}
\int_\Omega (\Ov {\vr \ln \vr}- \vr \ln \vr)(\tau,\cdot) \dx \leq \frac{\Lambda}{2\mu +\lambda} (1+R_0) \int_0^\tau \int_\Omega (\Ov {\vr \ln \vr}- \vr \ln \vr) \dx \dt 
\end{equation}
for any $\tau \in \overline I$. Hence, by Gronwall lemma and due to convexity of $\vr\mapsto\vr\ln\vr$,
$$
\mbox{for all $\tau\in\overline I$},\;\int_\Omega\Big(\Ov {\vr \ln \vr}- \vr \ln \vr\Big)(\tau,\cdot)\,{\rm d}x = 0.
$$
As $\vr\mapsto \vr\ln\vr$ is even strictly convex, this implies $\Ov {\vr \ln \vr}=\vr \ln \vr$, whence a.a. convergence (\ref{a.a.ep}) and consequently strong convergence in $L^q(I\times \Omega)$ for any $1\leq q<B+1$. 

The proof of the strong convergence is now finished and we get the momentum equation in the form
\begin{equation} \label{eq6.9}
\begin{aligned}
&\int_0^T \int_\Omega \big((\vr+Z)\vu \cdot \partial_t \vcg{\varphi} + (\vr+Z) (\vu\otimes \vu): \Grad \vcg{\varphi} + \Pi_\delta(\vr,Z)\, \Div \vcg{\varphi}\big) \dx \dt \\
= &\int_0^T\int_\Omega (\mu \Grad \vu :\Grad \vcg{\varphi} + (\mu+\lambda) \Div \vu \, \Div \vcg{\varphi} \big) \dx \dt - \int_\Omega \vc{m}_0 \cdot \vcg{\varphi}(0,\cdot) \dx.
\end{aligned} 
\end{equation}  
 
\subsection {Limit passage $\delta \to 0^+$} \label{se7}

First recall that we still have the energy inequality \eqref{eq6.5}. This inequality provides us only the estimates
\begin{equation} \label{eq7.1}
\begin{aligned}
&\|(\vr_\delta+Z_\delta)|\vu_\delta|^2\|_{L^\infty(I;L^1(\Omega))} + \|\vu_\delta\|_{L^2(I;W^{1,2}(\Omega))} + \|\vr_\delta\|_{L^\infty(I;L^\gamma(\Omega))} \\
& +  \|Z_\delta\|_{L^\infty(I;L^{q_{\gamma,\beta}}(\Omega))} + \delta^{1/B} (\|\vr_\delta\|_{L^\infty(I;L^B(\Omega))} +\|Z_\delta\|_{L^\infty(I;L^B(\Omega))}) \leq C,
\end{aligned}
\end{equation}
with number C independent of $\delta$, and $q_{\gamma,\beta}$  defined in Theorem \ref{t1}. We also deduce from Proposition \ref{L2} that the quantity $s_\delta$ defined according to (\ref{conv})$_{\vr_\delta,Z_\delta}$ is bounded:
\begin{equation}\label{s!!}
\mbox{for all $\tau\in\overline I$,}\; 0\le s_\delta(\tau)\le\overline a\;\mbox{for
a.a. $x\in \Omega$.}
\end{equation}

Estimate (\ref{eq7.1}) provides us solely the $L^\infty(I;L^1(\Omega))$ bound
for the pressure
which is not sufficient to pass to the limit in the momentum equation. To circumvent this problem, we need, similarly as in the previous limit passage, improved estimates of the densities.  We use as test function in the momentum equation \eqref{eq6.9} the function
$$
\vcg{\varphi} = \psi(t)\mathcal{B}\Big(\vr_\delta^\Theta -\frac{1}{|\Omega|} \int_\Omega \vr_\delta^\Theta\Big)
$$
for suitable $\Theta >0$ and $\psi \in C^1_c([0,T))$, $\psi(0)=0$, as in the previous section. For $\Theta$ sufficiently small (with respect to $\gamma$) it can be justified to be an appropriate test function. We get, similarly as for the compressible Navier--Stokes equations, the following
\begin{equation} \label{eq7.2}
\begin{aligned}
&\int_0^T \psi \int_\Omega P_\delta(\vr_\delta,Z_\delta) \vr_\delta^\Theta \dx \dt + \delta \int_0^T\int_\Omega (\vr_\delta^B + Z_\delta^B )\vr_\delta^\Theta \dx\dt \\
& = \int_0^T \psi \Big(\int_\Omega P_\delta (\vr_\delta,Z_\delta) \dx  + \delta \int_\Omega \Big(\vr_\delta^B + Z_\delta^B \\
& + \frac 12 \vr_\delta^2 Z_\delta^{B-2}+\frac 12 Z_\delta^2 \vr_\delta^{B-2}\Big) \dx \Big)\int_\Omega \vr_\delta^\Theta\dx \dt \\
& + \int_0^T \psi \int_\Omega \Big(\mu \Grad \vu_\delta : \Grad \vcg{\varphi} + (\mu+\lambda) \Div \vu_\delta \Div \vcg{\varphi}\\
&- (\vr_\delta + Z_\delta)(\vu_\delta \otimes \vu_\delta):\Grad \vcg{\varphi}\Big) \dx \dt \\
&- \int_0^T \partial_t \psi  \int_\Omega  \vr_\delta \vu_\delta \cdot \mathcal B \Big(\vr_\delta^\Theta -\frac{1}{|\Omega|} \int_\Omega \vr_\delta^\Theta\dx \Big) \dx \dt \\
&-\psi(0)\int_\Omega\vc{m}_0 \cdot \mathcal B\Big((\vr^\delta_0)^\Theta -\frac{1}{|\Omega|} \int_\Omega (\vr^\delta_0)^\Theta \dx\Big) \dx \\
&  + \int_0^T \psi \int_\Omega  \vr_\delta \vu_\delta \cdot  \mathcal B \Big(\Div (\vr_\delta^\Theta \vu_\delta) \\
&+(\Theta-1)\Big(\vr_\delta^\Theta \Div \vu -\frac{1}{|\Omega|} \int_\Omega \vr_\delta^\Theta \Div \vu_\delta\dx \Big)\Big)\dx \dt.  
\end{aligned}
\end{equation}

As in the case of the mono-fluid compressible Navier--Stokes equations we show that for $\Theta \leq \min\{\frac 23 \gamma-1,\frac{\gamma}{2}\} =: \gamma_{BOG}$, the right-hand side of this identity is bounded from above uniformly with respect to $\delta$, due to  estimates (\ref{eq7.1}), (\ref{s!!}). Therefore, using \eqref{eq2.2} from Hypothesis (H3) (and recalling (\ref{eq3.1})) we end up with
$$
\int_0^T  \int_\Omega \Big(\vr_\delta^{\gamma +\Theta} +\delta(\vr_\delta^{B+\Theta} + \vr^\Theta Z_\delta^B) \Big)\dx \dt 
\le C.
$$ 
Repeating the same with the test function 
$$
\vcg{\varphi} = \psi(t)\mathcal{B}\Big(Z_\delta^\Theta -\frac{1}{|\Omega|} \int_\Omega Z_\delta^\Theta\Big)
$$
we get similarly
$$
\int_0^T  \int_\Omega Z_\delta^{\beta + \Theta} \dx \dt 
\le C
$$
for $\Theta \leq \gamma_{BOG}$ if $\underline{a} =0$ or $\beta \leq \gamma$, and $\Theta \leq \beta_{BOG}$ if $\underline{a} >0$ and $\beta >\gamma$. 
Recalling (\ref{s!!}), we resume
\begin{equation}\label{eq7.3}
\|\vr_\delta\|_{L^\infty(I;L^{\gamma+\gamma_{\rm BOG}}(\Omega))}+
\|Z_\delta\|_{L^\infty(I;L^{{q_{\gamma,\beta}+\gamma_{\rm BOG}}}(\Omega))} \le C,
\end{equation}
\begin{equation} \label{eq7.3a}
\delta \int_0^T \psi \int_\Omega (\vr^{B +\gamma_{BOG}} + Z^{B +\gamma_{BOG}}) \dx \dt \leq C,
\end{equation}
\begin{equation} \label{eq7.4}
\| P(\vr_\delta,z_\delta)\|_{L^q(I\times \Omega)} \leq C
\end{equation}
for some $q>1$, and
\begin{equation} \label{eq7.3aa}
\delta \int_0^T \psi \int_\Omega Z^{B +\beta_{BOG}} \dx \dt \leq C
\end{equation}
if $\beta >\gamma$ and $\underline{a} >0$. 

{With estimates (\ref{eq7.1}) it is rudimentary to show
by nowadays standard arguments similarly as in (\ref{conv1}--\ref{conv3}) that 
$(\vrd,Z_\delta)\to(\vr,Z)$ in $C_{\rm weak}(\overline I; L^\gamma(\Omega))$, $
\vud\rightharpoonup\vu$ in $L^2(I;W^{1,2}(\Omega))$, $(\vrd\vu_\delta,Z_\delta\vud)\rightharpoonup_*(\vr\vu,Z\vu)$ in 
$L^\infty( I; L^{\frac{2\gamma}{\gamma+1}}(\Omega))$, $(\vrd+Z_\delta)\vud\to(\vr+Z)\vu$ in $C_{\rm weak}(\overline I; L^{\frac{2\gamma}{\gamma+1}}(\Omega)).$
}

We further write 
$$
P_\delta(\vr_\delta,Z_\delta)=-\eta_\delta(\sqrt{\vr_\delta^2+Z_\delta^2})P(\vr_\delta,Z_\delta)+P(\vr_\delta,Z_\delta),
$$
where
$$
\|\eta_\delta(\sqrt{\vr_\delta^2+Z_\delta^2})P(\vr_\delta,Z_\delta)\|_{L^q(Q_T)}\to 0.
$$
Moreover, by virtue of Hypothesis (H4) (cf. (\ref{eq2.4})) and (\ref{s!!}) 
$$
P(\vr_\delta,Z_\delta)=
P(\vr_\delta,\vr_\delta s_\delta)-P(\vr_\delta,\vr_\delta s)+
{\cal P}(\vr_\delta, s) +{\cal R}(\vr_\delta, s),
$$
where $s$ is defined in (\ref{sn}) in agreement with (\ref{conv}) (and it is also a \linebreak $C_{\rm weak^-*}(\overline I;L^\infty(\Omega))$ limit of $s_\delta$) and ${\cal P}$ and ${\cal R}$ possess all properties described in Hypothesis (H4), cf. (\ref{eq2.4}). According to Proposition \ref{L2} and due to Hypothesis (H3) --- namely equation (\ref{eq2.5-}) or (\ref{eq2.5}), respectively\footnote{Here we need that $\underline{a}>0$ if $\beta \geq \gamma +\gamma_{BOG}$.}
$$
\|P(\vr_\delta,\vr_\delta s_\delta)-P(\vr_\delta,\vr_\delta s)\|_{L^q(Q_T)}\to 0.
$$

Resuming (\ref{eq7.3}--\ref{eq7.4}) and the above considerations, we infer
\begin{equation}\label{Pidelta}
\Pi_\delta(\vr_\delta,Z_\delta)  \rightharpoonup \Ov{\Ov{ P(\vr,\vr s)}} = \Ov{P(\vr,Z)}\;\mbox{in $L^q(I\times \Omega)$ with some $q>1$.}
\end{equation}
{Here and in the sequel $\overline{g(\vr, Z,\vu,\Grad\vu)}$ denotes the $L^1(Q_T)$ weak limit of the sequence
$g(\vrd, Z_\delta,\vu_\delta,\Grad\vu_\delta)$ while  $\overline{\overline{g(\vr, \vr s,\vu,\Grad\vu)}}$ denotes the $L^1(Q_T)$ weak limit of the sequence
$g(\vrd, s \vr_\delta,\vu_\delta,\Grad\vu_\delta)$.}

Thus we may as above pass to the limit in the continuity equations and the momentum equation to get
\begin{equation} \label{eq7.5}
\begin{aligned}
&\int_0^T \int_\Omega \big((\vr+Z)\vu \cdot \partial_t \vcg{\varphi} + (\vr+Z) (\vu\otimes \vu): \Grad \vcg{\varphi} + \overline{{\Ov{P(\vr,\vr s)}}}\, \Div \vcg{\varphi}\big) \dx \dt \\
= &\int_0^T\int_\Omega (\mu \Grad \vu :\Grad \vcg{\varphi} + (\mu+\lambda) \Div \vu \, \Div \vcg{\varphi} \big) \dx \dt - \int_\Omega \vc{m}_0 \cdot \vcg{\varphi}(0,\cdot) \dx
\end{aligned} 
\end{equation} 
 for any $\vcg{\varphi} \in C^1_c([0,T)\times \Omega;R^3)$, 
\begin{equation} \label{eq7.6}
\begin{aligned}
\int_0^T \int_\Omega \big(\vr \partial_t \psi + \vr \vu \cdot \Grad \psi\big) \dx \dt  + \int_\Omega \vr_0 \psi(0,\cdot) \dx  & = 0,\\
\int_0^T \int_\Omega \big(Z \partial_t \psi + Z \vu \cdot \Grad \psi\big) \dx \dt + \int_\Omega Z_0 \psi(0,\cdot) \dx & = 0
\end{aligned}
\end{equation}
for any $\psi \in C^1_c([0,T) \times \Ov{\Omega})$. 

As $\gamma\ge 9/5$, the improved estimates of density guarantee $(\vr,Z)\in L^2(Q_T)$.
Consequently, according Proposition  \ref{p1} and its multi-dimensional version, Proposition \ref{p2}, the continuity equations (\ref{eq7.6})
are satisfied in the renormalized sense.

To conclude, we need to verify that the density sequence 
$\vr_\delta \to \vr$ a.a. in $Q_T$ {(and $Z_\delta \to Z$ a.a. in $Q_T$)}. We follow the approach presented in the previous section with one important change. Due to lower integrability of the pressure we are not any more able to deduce the effective viscous flux identity in the form \eqref{eq6.8}.

Instead, we shall replace Proposition \ref{p5} by
\begin{prop}\label{p5+}
Identity
\begin{equation} \label{eq7.9}
\Ov{\Ov{P(\vr,\vr s)T_k(\vr)}} -(2\mu+\lambda) \Ov{T_k(\vr) \Div \vu} = \Ov{\Ov{P(\vr,\vr s)}}\,  \Ov {T_k(\vr)} -(2\mu+\lambda) \Ov {T_k(\vr)} \Div \vu 
\end{equation} 
holds a.a. in $I\times \Omega$, where the truncation function $T_k(\vr)$ of $\vr$ is defined in (\ref{Lk}).
\end{prop}

\begin{proof}
The proof of this proposition follows the same lines as the proof of Proposition \ref{p5}
and is similar to the mono-fluid case, see e.g. \cite[Chapter 7]{NoSt}. The test functions for the momentum equation (\ref{eq6.9})$_{\vr_\delta,Z_\delta,\vu_\delta}$
are not those of (\ref{test1}) but $\varphi=\psi\nabla \Delta^{-1}(\phi T_k(\vr_\delta))$ and test functions for the limiting momentum equation \eqref{eq7.5} are not those of (\ref{test2}) but $\varphi=\psi\nabla \Delta^{-1}(\phi\Ov{T_k(\vr)})$. We obtain instead of  (\ref{ddd!})
\begin{equation}\label{ddd!1}
\int_0^T\intO{\psi\phi\Big(\Ov{\overline{P(\vr,\vr s)}}
-(2\mu +\lambda){\rm div}\vu\Big)\,\Ov{T_k(\vr)}}\,{\rm d}t
\end{equation}
$$
-\int_0^T\intO{\psi\phi\Big(\Ov{\overline{P(\vr,\vr s)\,T_k(\vr)}}
-(2\mu +\lambda)\overline{T_k(\vr)\,{\rm div}\,\vu}\Big)}\,{\rm d}t
$$
$$
=
\int_0^T\intO{\psi\vu\cdot\Big(\Ov{T_k(\vr)}\, {\tn R}\cdot((\vr+Z)\vu\phi)-(\vr+Z)\vu\cdot{\tn R}(\Ov{T_k(\vr)}\,\phi)\Big)
}\,{\rm d}t
$$
$$
-
\lim_{\delta\to 0}\int_0^T\intO{\psi\vu_\delta\cdot\Big(T_k(\vr_\delta) {\tn R}\cdot(\vr_\delta+Z_\delta)\vu_\delta\phi)-(\vr_\delta+ Z_\delta)\vu_\delta\cdot{\tn R}(T_k(\vr_\delta)\phi)\Big)
}\,{\rm d}t.
$$
During the derivation of (\ref{ddd!1}) we use several times the fact
that $T_k(\vr_\delta)$ verifies continuity equation (\ref{eq6.3}) in the renormalized sense with renormalizing function $b(\vr)=T_k(\vr)$, and $\Ov{T_k(\vr)}$ verifies the weak limit of the same equation exactly in the same way as in the mono-fluid case, cf. Proposition \ref{p1}. The right-hand side of (\ref{ddd!1}) tends to $0$ as 
$\delta\to 0$. This can be proved exactly in the same way as in Proposition
\ref{p5}. Indeed, the first line in (\ref{cvep!}) can be replaced by
$$
\mbox{for all $t\in \overline I$,}\;(T_k(\vr_\delta), T_k(Z_\delta))\to
(\overline{T_k(\vr)},\overline{T_k(Z)})\;\mbox{in $(C_{\rm weak}(I;L^B(\Omega))^2$}
$$
with {\it any} $1\le B<\infty$, and due to this gain of summability,	the rest of the proof follow  the same lines up to its end.
\end{proof}


To finish, we have to prove the strong convergence of the density. We proceed similarly as in the previous section, we just change the proof in order to use the effective viscous flux in the form \eqref{eq7.9}. 
Coming back to estimates (\ref{eq7.1}) we easily find out that
\begin{equation}\label{estTk}
\|T_{k}(\vr_\delta)\|_{L^r(Q_T)}\le c\qquad \text{ for } \, 1\le r\le\gamma+\gamma_{\rm BOG},
\end{equation}
$$
 \|L_k(\vr_\delta)\|_{L^\infty(I;L^{r}(\Omega))}\le c\qquad \text{ for } \, 1\le r<\gamma,
$$
uniformly with $\delta$ and $k$.

Estimates (\ref{estTk})  are enough for what we need provided $\gamma>9/5$. If $\gamma=9/5$ then $\gamma+\gamma_{BOG}=2$ and $\vr\in L^2(Q_T)$ together with $\vu$ still satisfy continuity equation in the renormalized sense by virtue of Proposition \ref{p1}. In spite of this fact, in this borderline case we need a better information.

It is provided in the following proposition:
\begin{prop}\label{odm}
The sequence $\vr_\delta$ satisfies
\begin{equation}\label{odm1}
{\rm osc}_{\gamma+1}[\vrd\rightharpoonup\vr](Q_T):=\sup _{k>1}\limsup_{\delta\to 0}
\int_{Q_T}|T_k(\vr_\delta)-T_k(\vr)|^{\gamma+1}\, {\rm d}x\, {\rm d} t<\infty.
\end{equation}
\end{prop}
\begin{proof}
Proposition
 \ref{odm} follows from the effective viscous flux identity derived in Proposition \ref{p5+}. To see this fact, we employ in (\ref{eq7.9}) decomposition (\ref{?!}) in order to get (recall, $\underline f$ was defined in (\ref{?!}))
$$
 \underline f \int_0^T\int_\Omega \Big(\overline{\vr^\gamma T_k(\vr)}-\overline{\vr^\gamma}\;\overline{T_k(\vr)}\Big)\, {\rm d}x\, {\rm d}t + \int_0^T\intO{\Big(
{\Ov{\overline{\pi(\vr,s)T_k(\vr)}}-\Ov{\overline{\pi(\vr,s)}}\;\overline{T_k(\vr)}}\Big)}\, {\rm d}t
$$
{
\begin{equation}\label{last}
=
\limsup_{\delta\to 0} \sum_{i=1}^3 I_\delta^i,
\end{equation}
}
where
$$
I_\delta^1=(2\mu+\lambda)\int_0^T\intO{\Big(T_k(\vrd)-{T_k(\vr)}\Big)\Div\vud}\,{\rm d}t,
$$
$$
I_\delta^2=(2\mu+\lambda)\int_0^T\intO{\Big(T_k(\vr)-\overline{T_k(\vr)}\Big)\Div\vud}\,{\rm d}t,
$$
$$
I_\delta^3= \int_0^T\intO{\Big(\Ov{\overline {\mathfrak{R}(\vr,s) T_k(\vr)}}-\Ov{\overline{\mathcal{R}(\vr,s)}}\,\overline{T_k(\vr)}\Big)}\,{\rm d}t.
$$
We first observe that the second integral at the left hand side is non negative (indeed, {the map $\vr\mapsto\pi(\vr,s)$ is non-decreasing
in $\vr$ for any fixed $s$} and we
can use Proposition \ref{p3}). Second, we employ the H\"older inequality and interpolation together with the lower weak 
semi-continuity of norms and bounds (\ref{eq7.1}),  (\ref{estTk})
to estimate integrals $I_\delta^1$, $I_\delta^2$ in order to get
{\begin{equation}\label{dod8}
|I_\delta^1+I_\delta^2|
\le c  \Big[{\rm osc}_{\gamma+1}[\vrd\rightharpoonup\vr](Q_T)\Big]^{\frac{1}{2\gamma}}
\end{equation}
with $c>0$ independent of $k$.
}
Finally, since $\mathcal{R}$ is continuous with compact support, integral $|I^3_\delta|$ is bounded by an universal constant $c=c(\mathcal{R})>0$.

We write,
$$
\int_0^T\int_\Omega \Big(\overline{\vr^\gamma T_k(\vr)}-\overline{\vr^\gamma}\;\overline{T_k(\vr)}\Big)\,{\rm d}x\,{\rm d}t
$$
$$
=\limsup_{\delta\to 0}\int_0^T\intO{ \Big(\vrd^\gamma-\vr^\gamma\Big)\Big(T_k(\vrd)-T_k(\vr)\Big)}\,{\rm d}t
$$
$$
+
\int_0^T\intO{\Big(\vr^\gamma-\overline{\vr^\gamma}\Big)\Big(\overline{T_k(\vr)}-T_k(\vr)\Big)}\,{\rm d}t
$$
$$
\ge \limsup_{\delta\to 0}\int_0^T\intO{ \Big|T_k(\vrd)-T_k(\vr)\Big|^{\gamma+1}}\,{\rm d}t,
$$
where we have employed convexity of $\vr\mapsto\vr^\gamma$ and concavity of $\vr\mapsto T_k(\vr)$ on $[0,\infty)$, and algebraic inequality
$$
|a-b|^\gamma\le |a^\gamma-b^\gamma|\;\mbox{and}\;|a-b|\ge |T_k(a)-T_k(b)|, \; \ (a,b)\in [0,\infty)^2.
$$
Inserting the last inequality into { (\ref{last}) yields (in combination with estimates of integrals $I_\delta^1-I_\delta^3$)} the statement of Proposition \ref{odm}.
 \end{proof}

We know that continuity equation (\ref{eq6.3})$_{\vr_\delta,\vu_\delta}$ is satisfied, in particular,  in the renormalized sense with renormalizing functions  $b(\vr)=L_k(\vr)$, and that (\ref{eq6.3})$_{\vr,\vu}$ is satisfied in the renormalized sense with the same function $L_k(\vr)$, cf. Proposition \ref{p1}, namely (\ref{ce7}).
Using these equations with test function $\varphi=1$ (and noticing that
$z L_k'(z) -L_k(z) = T_k(z)$), we get, in particular,
\begin{equation} \label{eq7.11}
\begin{aligned}
\int_\Omega (L_k(\vr_\delta)-L_k(\vr))(\tau,\cdot)\dx & = \int_0^\tau\int_{\Omega}(T_k(\vr)\Div \vu -\Ov{T_k(\vr)}\Div \vu_\delta)\dx\dt \\
&+ \int_0^\tau\int_\Omega (\Ov{T_k(\vr)}-T_k(\vr_\delta))\Div \vu_\delta \dx \dt 
\end{aligned}
\end{equation}
for all $\tau\in\overline I$,
where $\overline{T_k(\vr)}$ is $L^r(Q_T)$ weak limit of $T_k(\vrd)$, cf. (\ref{estTk}).   
Using \eqref{eq7.9}, decomposition (\ref{eq2.4}) and passing with $\delta \to 0$ we can write
\begin{equation} \label{eq7.12}
\begin{aligned}
&\int_\Omega (\Ov{L_k(\vr)}-L_k(\vr))(\tau,\cdot)\dx = \int_0^\tau\int_{\Omega}(T_k(\vr) -\Ov{T_k(\vr)})\Div \vu\dx\dt 
\\
&+ \frac{1}{2\mu+\lambda} \int_0^\tau\int_\Omega \big(\Ov{\Ov{{\cal P}(\vr,s)}}\,  \Ov {T_k(\vr)}
-\Ov{\Ov{{\cal P}(\vr, s)T_k(\vr)}}\big) \dx \dt\\
&+\frac{1}{2\mu+\lambda} \int_0^\tau\int_\Omega \big(\Ov{\Ov{{\cal R}(\vr,s)}}\,  \Ov {T_k(\vr)}
-\Ov{\Ov{{\cal R}(\vr, s)T_k(\vr)}}\big) \dx \dt
\end{aligned} 
\end{equation}
for all $\tau\in\overline I$, where $\overline{L_k(\vr)}$ is $C_{weak}(\overline I;L^r(\Omega))$ limit of sequence $L_k(\vrd)$, cf. (\ref{estTk}) and Proposition \ref{p1}. The first term at the right-hand side is bounded by
$$
\|T_k(\vrd)-T_k(\vr)\|_{L^2(Q_T)}\|\Div\vu\|_{L^2(Q_T)}\le c\limsup_{\delta\to 0}
\|T_k(\vrd)-T_k(\vr)\|_{L^1(Q_T)}^{\frac {\gamma-1}{2\gamma}},
$$ 
where we have used estimates (\ref{eq7.1}), (\ref{odm1}) and interpolation.\footnote{
Estimate (\ref{odm1}) is needed only in the borderline case $\gamma=9/5$ when $\gamma+\gamma_{BOG}=2$. If $\gamma>9/5$ then $\gamma+\gamma_{BOG}>2$ and we can get similar result interpolating $T_k(\vrd)-T_k(\vr)$ between $L^1$ and $L^{\gamma+\gamma_{BOG}}$ while using only estimates (\ref{estTk}).}
The second term is non negative according to Proposition \ref{p3}. By the same argumentation as in (\ref {eq6.8a}), the third term is bounded by
$$
c\Lambda (1+R_0)\int_0^\tau\int_{\Omega}\Big(\overline{\vr\ln\vr}-\vr\ln\vr\Big)\,{\rm d} x\,{\rm d}t
$$
with sufficiently large $\Lambda>1$, where $R_0>1$ is such that $\cup_{s\in[\underline a,\overline a]}{\rm supp}{\cal R}(\cdot,s)\subset [0,R_0]$.

Writing
$$
\|T_k(\vr)-\overline{T_k(\vr)}\|_{L^1(Q_T)}\le \|T_k(\vr)-\vr\|_{L^1(Q_T)}+\|\overline{T_k(\vr)}-\vr\|_{L^1(Q_T)}
$$
$$
\le \|T_k(\vr)-\vr\|_{L^1(Q_T)}+ \liminf_{\delta\to 0}\|{T_k(\vr_\delta)}-\vr_\delta\|_{L^1(Q_T)}\to 0\;\mbox{as $k\to\infty$},
$$
and recalling
\begin{equation}\label{LT}
\overline{L_k(\vr)} \to\overline{\vr\ln\vr},\;L_k(\vr)\to\vr\ln\vr
,\;\mbox{in $C_{weak}([0,T];L^r(\Omega))$ for any $1\leq r<\gamma$}
\end{equation}
we arrive at (\ref{eqeq}) which implies in virtue of Gronwall lemma and strict convexity of $\vr\mapsto\vr\ln\vr$ on $[0,\infty)$
$$
\vrd\to\vr\;\mbox{a.a. in $Q_T$}.
$$
{This relation in combination with (\ref{cvs}) yields also $
Z_\delta\to Z$ a.e. in $Q_T$. Now, it is rather standard to pass to the limit in the energy inequality (\ref{eq6.5})$_{\vrd,Z_\delta,\vud}$ 
and to obtain energy inequality (\ref{eq2.9}).}

This completes the proof of Theorem \ref{t1}.

\section{The real bi-fluid system: Proof of Theorem \ref{t1bi}}\label{sdod1}

In this Section we prove Theorem \ref{t1bi}. In view of what was said in Section \ref{se1}, it reduces to show that system (\ref{eq1.1bi}--\ref{eq1.3bi}) can be viewed as system (\ref{eq1.1}--\ref{eq1.3}) with an adequate pressure. The task is therefore to verify that the new pressure satisfy  Hypotheses  (H3--H5) from Section \ref{se2}, and conclude.

We introduce new unknowns
\begin{equation}\label{bi1}
\vr=\mathfrak{a}\vr_+,\quad Z= (1-\mathfrak{a})\vr_-
\end{equation}
and use the last constraint in (\ref{eq1.3bi}), namely
$$
\mathfrak{P}_{+}(\mathfrak{ \vr_+}) = \mathfrak{P}_{-}(\vr_-)
$$
to express quantities $(\mathfrak{a}, {\vr_+}, \vr_-)$ in terms of
new quantities $(\vr, Z)$. Equation for $\vr_+$ reads
\begin{equation}\label{bi2}
\vr_+\mathfrak{q}(\vr_+)-\mathfrak{q}({\vr_+})\vr-Z\vr_+=0,\;\mbox{where
$\mathfrak{q}={\mathfrak{P}}^{-1}_{-}\circ\mathfrak{P}_{+}$}. 
\end{equation}
It admits for
any $\vr\ge 0$, $Z\ge 0$ a unique solution 
\begin{equation}\label{bid1}
\left\{\begin{array}{c}
0<\vr_+={\vr_+}(\vr,Z)\in[\vr,\infty)\;\mbox{if $\vr>0$ or $Z>0$},\\
\vr_+(0,0)=0
\end{array}\right\}
\end{equation}
such that
\begin{equation}\label{bid2}
{\vr_+}(\vr,0)=\vr,\quad {\vr_+}(0,Z)=\mathfrak{q}^{-1}(Z).
\end{equation}
Indeed, if $(\vr,Z)\in (0,\infty)^2$, we get this information employing monotonicity of  function
$s\mapsto s\mathfrak{q}(s)-\mathfrak{q}({s})\vr-Z s$ and applying to it the intermediate value theorem for continuous functions, and solving equation (\ref{bi2}) 
explicitly in all remaining cases $\vr=0$ or $Z=0$.

According to the implicit function theorem applied to (\ref{bi2}), the function ${\vr_+}\in C^2((0,\infty)^2)$ and
for all $ \vr>0$, $ Z>0$,
\begin{equation}\label{*!}
0<\partial_\vr\vr_+(\vr,Z)
=\frac{\vr_+\mathfrak{q}(\vr_+)} {\vr\mathfrak{q}(\vr_+)+\vr_+\mathfrak{q}'(\vr_+)(\vr_+-\vr)}
=\frac{Q'(\vr_+)}{\mathfrak{q}'(\vr_+)-\vr Q''(\vr_+)}
\end{equation}
and
\begin{equation}\label{*!+}
 0<
\partial_Z\vr_+(\vr,Z)= 
\frac{(\vr_+)^2}{\vr \mathfrak{q}(\vr_+)+\vr_+\mathfrak{q}'(\vr_+)(\vr_+-\vr)
},
\end{equation}
where  $Q(s)=\int_0^s\frac{\mathfrak{q}(z)}z\,{\rm d}z$. 

We derive from these formulas and (\ref{bid1}--\ref{bid2}), in particular, the following:
\begin{description}
\item {\it 1.} 
\begin{equation}\label{it1}
\forall Z\ge 0,\;\mbox{function $\vr\mapsto \vr_+(\vr,Z)$ is continuous in $[0,\infty)$},
\end{equation}
$$
\forall \vr\ge 0,\;\mbox{function $Z\mapsto \vr_+(\vr,Z)$ is continuous in $[0,\infty)$}.
$$
\item{\it 2.} 
\begin{equation}\label{it2}
\forall (\vr,Z)\in {\cal O}_{\underline a}\cap (0,\infty)^2 ,
\end{equation}
$$
\max\Big\{\vr\,,\,\underline c\Big(\vr+\mathfrak{q}^{-1}(Z)\Big)\Big\}\le\vr_+(\vr,Z)\le{\overline c}\big(\vr+\mathfrak{q}^{-1}(Z)\big),
$$
where $0<\underline c<\overline c<\infty$.

Indeed, expression for $\partial_\vr\vr_+(\vr,Z)$ is positive and admits upper bound
\begin{equation} \label{137a}
\frac{Q'(\vr_+)}{\mathfrak{q}'(\vr_+)-\vr_+ Q''(\vr_+)}=1,
\end{equation}
and lower bound
$$
0<\underline q \le\frac{\vr_+\mathfrak{q}(\vr_+)} {\vr_+\mathfrak{q}(\vr_+)+\vr_+\mathfrak{q}'(\vr_+)\vr_+}\le \frac{\vr_+\mathfrak{q}(\vr_+)} {\vr\mathfrak{q}(\vr_+)+\vr_+\mathfrak{q}'(\vr_+)(\vr_+-\vr)}
$$
according to assumption (\ref{Hbi4}), where we have used (\ref{bid1}--\ref{bid2}). This yields (\ref{it2}) after integration of (\ref{*!}) from $0$ to $\vr$, knowing that according to (\ref{bid1}--\ref{bid2}), $\vr$ is another lower bound.

\item{\it 3.} 
Returning with (\ref{it2}) to (\ref{*!+}), we get by the same token
\begin{equation}\label{it3}
\forall (\vr,Z)\in {\cal O}_{\underline a}\cap (0,\infty)^2),\;
 0<
\partial_Z\vr_+(\vr,Z)\le
\frac {\vr_+}{\vr}\frac{\vr_+}{\mathfrak{q}(\vr_+)
}
\end{equation}
$$
\le \overline C\Big(1+\frac {\mathfrak{q}^{-1}(\overline a\vr)}
\vr\Big) \times
\left\{\begin{array}{c}
\frac{\vr+ \mathfrak{q}^{-1}(\overline a\vr)}{\mathfrak {q}(\vr)}\;\mbox{if $\underline a=0$},\\
\frac{\vr+ \mathfrak{q}^{-1}(\overline a\vr)}{\mathfrak {q}(\underline{c} (\vr+\mathfrak{q}^{-1}(\underline a\vr)))}\;\mbox{if $\underline a>0$}
\end{array}\right\}
$$ 
with some $\overline C>0$.
\end{description}

After this preparation, we know that system (\ref{eq1.1bi}--\ref{eq1.3bi}) can be rewritten as system 
(\ref{eq1.1}--\ref{eq1.3}) with
\begin{equation}\label{biP}
P(\vr,Z):= \mathfrak{P}_{+}(\vr_+(\vr,Z)).
\end{equation}

Our goal is to show that pressure $P$ defined above verifies all Hypotheses
(H3--H5) with 
\begin{equation}\label{gb}
\gamma=\gamma^+,\; \beta=\gamma^-.
\end{equation}

Note that due to (\ref{Hbi2}) we have for large $s$:
\begin{equation} \label{asympt}
\begin{aligned}
\mathfrak{P}_+ (s) \sim_{\infty} s^{\gamma^+}, \quad & \mathfrak{P}_- (s) \sim_{\infty} s^{\gamma^-}, \\
\mathfrak{q}(s) \sim_{\infty} s^{\frac{\gamma^+}{\gamma^-}}, \quad & \mathfrak{q}^{-1}(s) \sim_{\infty} s^{\frac{\gamma^-}{\gamma^+}}. 
\end{aligned}
\end{equation}
Recalling again assumptions (\ref{Hbi2}) and (\ref{it2}), combining them with (\ref{asympt}), we deduce immediately (\ref{eq2.2}), i.e. $\forall (\vr,Z)\in {\cal O}_{\underline a}$
\begin{equation}\label{Pbi}
\underline C(\vr^{\gamma^+} + Z^{\gamma-} -1)
\le P(\vr,Z)\le \overline C(1+\vr^{\gamma^+}+Z^{\gamma^-})
\end{equation}
with some $0<\underline C<\overline C<\infty$.

Coming back
with this information to the explicit formula (\ref{HP}) for the Helmholtz 
function we infer
$$
\underline C(\vr^{\gamma^+}+ Z^{\gamma^-}-1)\le H_P(\vr,Z)\le \overline C(1+\vr^{\gamma_+}+Z^{\gamma^-}).
$$
This finishes the verification of formula (\ref{eq2.3}) related to growth conditions
(\ref{eq2.2}).

Now we employ (\ref{it2}), (\ref{it3}), monotonicity of $\mathfrak{q}$ and convexity of $\mathfrak{P}_+$
in order to deduce for all $(\vr,Z)\in {\cal O}_{\underline a}{\cap (0,\infty)^2}$:
\begin{multline*}
0<\partial_Z P(\vr,Z)=\mathfrak{P}_+'(\vr_+(\vr,Z))\partial_Z\vr_+(\vr,Z)\\
\le C \frac{\mathfrak P_{+}'(\vr+\mathfrak{q}^{-1}(\overline a \vr))(\vr+\mathfrak{q}^{-1}(\overline a \vr))^{2}}{\vr \mathfrak{q}(\vr)}
\end{multline*}

By the same token, recalling assumptions (\ref{Hbi2})  together with (\ref{asympt}) for estimates at infinity  and (\ref{Hbi3+}) for estimates near zero, we 
deduce from the latter formula
\begin{equation}\label{PZbi}
\forall (\vr,Z)\in {\cal O}_{\underline a}\cap (0,\infty)^2,\;0<\partial_Z\vr_+(\vr,Z)\le
C\Big(\vr^{-\underline\Gamma}+\vr^{\overline\Gamma-1}\Big),
\end{equation}
where 
$$
\overline\Gamma=
\left\{\begin{array}{c}
\max\{\gamma^+ -\frac{\gamma^+} {\gamma^-}+ 1,
\, \gamma^- + \frac{\gamma^-}{\gamma^+}-\frac{\gamma^+} {\gamma^-}\}\;\mbox{if $\underline a=0$}\\
\max\{\gamma^+ -\frac{\gamma^+} {\gamma^-}+ 1, \gamma^- +\frac{\gamma^-} {\gamma^+}-1 
 \}\;\mbox{if $\underline a>0$.}
\end{array}
\right\}.
$$

 The new pressure (\ref{biP}) satisfies all requirements
of  Hypothesis (H3).


Concerning Hypothesis (H4), we start with the first one, namely (\ref{eq2.4}). It suffices to take  $ {\cal P}(\vr,s)= P(\vr,s \vr)$ and
$\mathcal{R}=0.$  If $\gamma^+=9/5$, we need also the decomposition (\ref{?!}).
We show it for any $\gamma^+\ge 9/5$ provided assumption (\ref{Hbi4}) is valid.
We take
\begin{equation}\label{!!}
P(\vr, \vr s)=\underline q\frac{\underline a_+} {2\gamma^+}\vr^{\gamma^+}+\pi(\vr,s) -\mathcal{R}(\vr,s)
\end{equation}
with
 $$
\pi(\vr,s)=P(\vr,\vr s)-\underline q\frac{\underline a_+} {2\gamma}\vr^{\gamma^+} + \underline q b_+\zeta(\vr)\min\{\underline r,\vr\},
$$
$$
\mathcal{R}(\vr,s)=\underline q b_+\zeta(\vr)\min\{\underline r,\vr\}, 
$$
where
$\underline r$ solves equation $a_+ z^{\gamma-1} -2b_+=0$ and $\zeta\in C^2_c([0,\infty))$, $\zeta(z)=1$ for $z\in [0,R)$, $\zeta(z)=0$ for $z\in [3R,\infty)$, $0\le -\zeta'(z)\le \frac 1R$ with $R\ge\max\{2\underline r, (2b_+/\underline a_+)^{1/\gamma^+}\}$.
Indeed, for all $s\in [\underline a,\overline a]$ and all $\vr >0$, we have
$$
\frac{\partial\pi}{\partial\vr}(\vr,s)=
\mathfrak{P}_+'(\vr_+(\vr,\vr s))\Big(\frac{\partial\vr_+}{\partial\vr}(\vr,\vr s)+s\frac{\partial\vr_+}{\partial Z}(\vr, \vr s)\Big) -\underline q\frac{\underline a_+}2\vr^{\gamma^+-1}
$$
$$
+\left\{
\begin{array}{c}
\underline q b_+\;\mbox{if $\vr\in [0,\underline r]$}\\
\underline q b_+\zeta'(\vr)\;\mbox{if $\vr>\underline r$}
\end{array}
\right\}
$$ 
$$
\ge
\underline q\underline a_+ \vr^{\gamma^+-1} -\underline q b_+ -\underline q\frac{\underline a_+}2\vr^{\gamma^+-1}
+\left\{
\begin{array}{c}
\underline q b_+\;\mbox{if $\vr\in [0,\underline r]$}\\
0\;\mbox{if $\underline r<\vr<R$}\\
\underline q\underline r b_+\zeta'(\vr)\;\mbox{if $\vr>R$}
\end{array}
\right\}\ge 0,
$$
where we have used first (\ref{*!}), (\ref{*!+}) and then assumption (\ref{Hbi2}).
To be more exact we should take in the formulas above instead of $\vr\mapsto m(\vr):= \min\{\underline r,\vr\}$ its regularization via a one dimensional mollifier (e.g. $m*\chi_{\underline r/8}$, cf. (\ref{chi}), where we would consider $m$ extended by $\underline r$ to negative values of $\vr$), in order to comply with the requested regularity of ${\cal R}$. We skip this unessential technical point letting the details to the interested reader. { Finally, coming back to (\ref{it2}) and (\ref{biP}),
we easile verify assumption (\ref{?!+}).

It remains to verify Hypothesis (H5) which consists in verifying (\ref{eq2.3a-}).
This amounts to show that supremum of $|\partial_\vr P(s,Z)|+ |\partial^2_Z P(s,Z)|$ over the set $\underline r\le s\le \vr$, $\underline a\vr\le Z\le \overline a\vr$
has a uniform polynomial growth in $\vr$  for large $\rho$'s with coefficient
$c=c(\underline r)$ which may blow up as $\underline r\to 0$.
We calculate first
$$
\partial_\vr P(\vr,Z)= \mathfrak{P}'(\vr_+(\vr,Z))\partial_\vr\vr_+(\vr,Z)
$$
and conclude thanks to (\ref{*!}), (\ref{137a}) and assumption (\ref{Hbi2}). Likewise, we calculate
$$
\partial^2_Z P(\vr,Z)=\mathfrak{P}''(\vr_+(\vr,Z))\Big(\partial_Z\vr_+(\vr,Z)\Big)^2
+\mathfrak{P}'(\vr_+(\vr,Z))\partial^2_Z\vr_+(\vr,Z)
$$
and look at (\ref{*!+}--\ref{it3}) and assumption (\ref{Hbi2}).

We have thus shown that problem to solve (\ref{eq1.1bi}--\ref{eq1.3bi}) reduces
to solving of problem (\ref{eq1.1}--\ref{eq1.3}) with new pressure (\ref{biP}).
Thorem \ref{t1bi} is proved.


\section{The academic multifluid system} \label{se8}

In this section we formulate the problem for the case of academic multifluid system in the spirit of the model bi-fluid system studied in the previous sections. Note that we shall not treat the real multifluid problem as due to its complexity it would require many more details which would extend the length of the paper. On the other hand, in case of the academic multifluid system the extension is much more straightforward. We consider the following system of equations ($K\geq 2$)
\begin{equation}\label{eq8.1}
\begin{aligned}
\partial_t \vr  &+ \Div(\vr \vu) = 0, \\
\partial_t Z_i &+ \Div (Z_i \vu)  = 0,  \quad i=1,2,\dots, K,\\
\partial_t \big((\vr+\sum_{i=1}^K Z_i)\vu\big) &+ \Div\big((\vr+\sum_{i=1}^K Z_i) \vu\otimes \vu) + \nabla P(\vr,Z_0,Z_1,\dots, Z_K) \\
= & \mu \Delta \vu + (\mu+\lambda)\Grad \Div \vu,
\end{aligned}
\end{equation}
together with the boundary condition
\begin{equation} \label{eq8.2}
\vu = \vc{0}
\end{equation}
on $(0,T)\times \partial \Omega$, and the initial conditions in $\Omega$
\begin{equation} \label{eq8.3}
\begin{aligned}
\vr(0,x) &= \vr_0(x), \\
Z_i(0,x) &= Z_{i0}(x), \quad i=1,2,\dots, K,\\
\Big(\vr + \sum_{i=1}^K Z_i\Big)\vu(0,x) &= \vc{m}_0(x).
\end{aligned}
\end{equation}

The weak formulation of this problem  can be easily deduced by modifying Definition \ref{d1}. We replace our Hypotheses (H1--H5) from the bi-fluid flow to the following ones in the case of the multifluid system.
\\ \\
\noindent {\bf Hypothesis (MH1).}
\begin{equation} \label{eq2.1MF}
\begin{aligned}
(\vr_0,Z_{10}, Z_{20}, \dots, Z_{K0})\in & {\cal O}_{\underline {\vec {a}}} :=\\
\{(\vr,Z_1,Z_2,\dots, Z_K)\in R^{K+1}& \,|\,\vr\in [0,\infty),\underline {a}_i \vr \le Z_i\le \overline {a}_i \vr, \, i=1,2,\dots, K\},
\end{aligned}
\end{equation}
where $0\leq \underline{a}_i<\overline{a}_i <\infty$, $i=1,2,\dots, K$. 

Next hypothesis deals with the integrability of the initial conditions.  
\\ \\
{\bf Hypothesis (MH2).}
\begin{equation} \label{eq2.6MF}
\vr_0 \in L^\gamma(\Omega), \; Z_{i0} \in L^{\beta_i}(\Omega) \; \text{ if } \beta_i > \gamma,
\end{equation}
$$
\vc{m}_0 \in L^1(\Omega;R^3),\; (\vr_0+\sum_{i=1}^K Z_{i0})|\vu_0|^2\in L^1(\Omega), \, i=1,2,\dots, K.
$$

As above, we denote
$H_P(\vr,Z_1,\dots, Z_K)$ the solution to the partial differential equation of the first order in $(R^+)^{K+1}$
\begin{multline}\label{MFeqH}
P(\vr,Z_1,Z_2,\dots, Z_K) = \vr \pder{H_P(\vr,Z_1,Z_2,\dots, Z_K)}{\vr} \\ + \sum_{i=1}^K Z_i\pder{H_P(\vr,Z_1,Z_2,\dots, Z_K)}{Z_i}-H_P(\vr,Z_1,Z_2,\dots, Z_K).
\end{multline}

Similarly as in the bi-fluid system, we can construct a suitable solution to (\ref{MFeqH}) explicitly, by means of the following integral formula
\begin{equation} \label{MFsolH}
H_P(\vr,Z_1,Z_2,\dots, Z_K) = \vr \int_{1}^{\vr} \frac{P\big(s, s\frac{Z_1}{\vr}, s\frac{Z_2}{\vr}, \dots, s\frac{Z_K}{\vr}\big)}{s^2} \, {\rm d}s, 
\end{equation}
if $\vr >0$; $H_P(0,0,\dots, 0) =0$.

Next hypotheses deal with the form of the pressure function. We distinguish the case when the pressure (see Hypothesis (MH3) below) is bounded from above by $Z_i^{\beta_i}$ for $\beta_i \geq \gamma+\gamma_{BOG}$ or $\beta_i < \gamma+\gamma_{BOG}$, where $\gamma_{BOG} = \min \{\frac 23 \gamma-1,\frac{\gamma}{2}\}$ is as in Section \ref{se2}. To simplify the notation, we also denote
\begin{equation} \label{8.5}
J= \{j\in \{1,2,\dots,K\}; \beta_j \geq \gamma+\gamma_{BOG}\}.
\end{equation}

We assume
\\ \\
{\bf Hypothesis (MH3).}\\ \\
The function $P$: $[0,\infty)^{K+1} \to [0,\infty)$ is a continuous function in $[0,\infty)^{K+1}$, $P(0,0,\dots,0) =0$, $P\in C^1((0,\infty)^{K+1})$, and
\begin{equation} \label{eq8.7}
\underline{C}(\vr^\gamma + \sum_{i=1}^K Z_i^{\beta_i} +1) \leq P(\vr,Z_1,\dots, Z_K) \leq \overline{C}(\vr^\gamma + \sum_{i=1}^K Z_i^{\beta_i} +1)
\end{equation}
with $\gamma \geq \frac 95$, $\beta_i >0$, $i=1,2,\dots, K$. 
We moreover assume for $i=1,2,\dots, K$
\begin{equation}\label{MFeq2.5-}
|\partial_{Z_i}P(\vr,Z_1,Z_2,\dots, Z_K)|\le C(\vr^{-\underline\Gamma}+\vr^{\overline\Gamma-1})\;\mbox{in ${\cal O}_{\underline {\vec{a}}}\cap (0,\infty)^{K+1}$}
\end{equation}
with some $0\le\underline\Gamma<1$, and with some $0< \overline\Gamma < \gamma + \gamma_{BOG}$ if $\underline{a}_i=0$, 
$0<\overline\Gamma<{\rm max}\{\gamma+\gamma_{BOG}, \beta_i+(\beta_i)_{BOG}\} $ if $\underline{a}_i>0$. 

Next we have
\\ \\
{\bf Hypothesis (MH4).}\\ \\
We assume
%
\begin{equation} \label{MFeq2.4}
P(\vr,\vr s_1,\vr s_2\dots, \vr s_K)=
{\cal P}(\vr,s_1,s_2,\dots, s_K) - {\cal R} (\vr,s_1,s_2,\dots, s_K),
\end{equation}
where $[0,\infty)\ni\vr\mapsto {\cal P}(\vr,s_1,s_2,\dots, s_K)$ is non decreasing  for any $s_i\in [\underline {a}_i,\overline {a}_i]$, $i=1,2,\dots, K$, and $\vr\mapsto {\cal R}(\vr,s_1,s_2,\dots, s_K)$ is for any $s_i \in [\underline {a}_i,\overline {a}_i]$, $i=1,2,\dots, K$ a non-negative $C^2$-function  in $[0,\infty)$ uniformly bounded with respect to $s_i \in [\underline {a}_i, \overline {a}_i]$, $i=1,2,\dots, K$ with compact support uniform with respect to $s_i \in [\underline {a}_i, \overline {a}_i]$, $i=1,2,\dots, K$. Here, $\underline {a}_i, \overline {a}_i$ are the constants from relation (\ref{eq2.1MF}). Moreover, if $\gamma=\frac 95$, we suppose that
\begin{equation}\label{?!MF}
{\cal P}(\vr,s_1,s_2,\dots, s_K)=f(s_1,s_2,\dots, s_K)\vr^\gamma +\pi(\vr,s_1,s_2,\dots, s_K),
\end{equation}
where $[0,\infty)\ni\vr\mapsto \pi(\vr,s_1,s_2,\dots, s_K)$ is non decreasing  for any $s_i \in [\underline {a}_i, \overline {a}_i]$, $i=1,2,\dots, K$, $f\in L^\infty\big(\prod_{i=1}^K(\underline {a}_i,\overline {a}_i)\big)$, and
$$
{\rm ess\, inf}_{\prod_{i=1}^K(\underline {a}_i,\overline {a}_i)} f(s_1,s_2,\dots, s_K)\ge \underline {f}_M>0.
$$
Further, for all $\vr\in (0,1)$
$$
\sup_{s_i\in [\underline{a}_i, \overline{a}_i], i=1,2,\dots,K}P(\vr,\vr s_1,\dots,\vr s_K) \leq c \vr^\alpha
$$
for some $c>0$ and $\alpha >0$. 

Finally, we assume\\ \\
\vbox{\noindent 
{\bf Hypothesis (MH5).}\\ \\
Function $\vr\mapsto P(\vr,Z_1,Z_2,\dots,Z_K)$ is for all $Z_i>0$, $i=1,2,\dots, K$ locally Lipschitz on $(0,\infty)$  and function $Z_1,Z_2,\dots, Z_K\mapsto \partial_{Z_j} P(\vr,Z_1,Z_2,\dots, Z_K)$ is for all $\vr>0$ locally Lipschitz on $(0,\infty)^K$ for any $j=1,2,\dots, K$ with Lipschitz constant
\begin{equation}\label{eq2.3a-MF}
\widetilde L_{PM}(\vr,Z_1,Z_2,\dots, Z_K)\le C(\underline r)(1+\vr^A) 
\end{equation}
for all $\underline r>0$, $(\vr, Z_1,Z_2,\dots, Z_K)\in {\cal O}_{\underline a} \cap (\underline r,\infty)^{K+1}$ 
with some non negative number $A$. Number $C(\underline r)$ may diverge to $+\infty$ as $\underline r\to 0^+$.}
\\ \\

Following the proof for the bi-fluid system we can obtain the following result. Note, in particular, that we may apply Proposition \ref{p4} for each density $Z_i$, $i=0,1,\dots, K$ separately.

\begin{thm} \label{t2}
Let $\gamma \geq \frac 95$. Then under Hypotheses (MH1--MH5), there exists at least one weak solution to problem \eqref{eq8.1}--\eqref{eq8.3}. Moreover, the densities $\vr \in C_{weak}([0,T); L^\gamma (\Omega))$, $Z_i \in C_{weak}([0,T); L^{\max\{\gamma,\beta_i\}} (\Omega))$, $i=1,2,\dots, K$, $(\vr+\sum_{i=1}^K Z_i)\vu \in C_{weak}([0,T); L^q (\Omega;R^3))$ for some $q>1$, and $P(\vr,Z_1,Z_2,\dots, Z_K) \in L^q(\Omega)$ for some $q>1$.   
\end{thm}

{\bf Acknowledgement:} The work of M. Pokorn\'y was supported by the Czech Science Foundation, grant No. 16-03230S. Significant part of the paper was written during the stay of M. Pokorn\'y at the University of Toulon. The authors aknowledge this support.

\end{document}